\documentclass[10pt, a4paper]{amsart}
\usepackage[T1]{fontenc}
\usepackage{geometry}
\usepackage[dvipsnames]{xcolor}
\usepackage[hidelinks,colorlinks=true,citecolor=ForestGreen,linkcolor=RubineRed]{hyperref}
\usepackage[timezonesep={\ UTC}, showseconds=false]{datetime2}
\usepackage{bookmark}

\usepackage{amsmath}
\usepackage{amssymb}
\usepackage{amsfonts}
\usepackage{mathtools}
\usepackage{amsthm}
\usepackage{braket}
\usepackage{comment}
\usepackage{mathrsfs}
\usepackage{dsfont}
\usepackage{bbm}
\usepackage{stmaryrd}
\usepackage{tensor}
\usepackage{tikz-cd}
\usepackage{extarrows}
\usepackage{manfnt}

\newcommand{\Bcrys}{\mathbb{B}_\mathrm{cris}}
\DeclareFontFamily{U}{wncy}{}
\DeclareFontShape{U}{wncy}{m}{n}{<->wncyr10}{}
\DeclareSymbolFont{mcy}{U}{wncy}{m}{n}
\DeclareMathSymbol{\Sha}{\mathord}{mcy}{"58}

\newcommand{\cA}{\mathcal{A}}
\numberwithin{equation}{section}
\newtheorem{thm}{Theorem}[section]
\newtheorem{mainThm}{Theorem}

\newtheorem{prop}[thm]{Proposition}
\newtheorem{lem}[thm]{Lemma}
\newtheorem{cor}[thm]{Corollary}

\theoremstyle{remark}
\newtheorem{rem}[thm]{Remark}
\theoremstyle{definition}
\newtheorem{de}[thm]{Definition}

\newcommand{\numberset}{\mathbf}

\newcommand{\Z}{\numberset{Z}}
\newcommand{\Q}{\numberset{Q}}

\newcommand{\C}{\numberset{C}}

\newcommand{\Zp}{{\Z_p}}
\newcommand{\Qp}{{\Q_p}}
\newcommand{\Qpbar}{\overline{\Q}_p}
\newcommand{\Qbar}{\overline{\Q}}

\newcommand{\1}{\mathbf{1}}

\newcommand{\HIw}{H^1_\mathrm{Iw}}

\newcommand{\Dcris}{\mathbb{D}_\mathrm{cris}}

\newcommand{\cyc}{\mathrm{cyc}}

\DeclareMathOperator{\Hom}{Hom}

\DeclareMathOperator{\Gl}{GL}
\DeclareMathOperator{\gal}{Gal}
\DeclareMathOperator{\Gal}{Gal}

\DeclareMathOperator{\Fil}{Fil}

\DeclareMathOperator{\rank}{rank}
\DeclareMathOperator{\corank}{corank}
\DeclareMathOperator{\coker}{coker}
\DeclareMathOperator{\Char}{char}

\newcommand{\im}{\mathrm{im}}

\newcommand{\inj}{\hookrightarrow}
\newcommand{\surj}{\twoheadrightarrow}

\renewcommand{\injlim}{\varinjlim}
\renewcommand{\projlim}{\varprojlim}
\newcommand{\plim}{\projlim}
\newcommand{\ilim}{\injlim}

\DeclareMathOperator{\tr}{Tr}

\newcommand{\Sel}{\mathrm{Sel}}
\newcommand{\hone}{H^1}
\newcommand{\fcond}{\mathbf{f}}
\newcommand{\scond}{\mathbf{s}}
\newcommand{\honef}{\hone_\fcond}
\newcommand{\hones}{\hone_\scond}

\newcommand{\ur}{\mathrm{ur}}

\newcommand{\Col}{\mathrm{Col}}

\DeclareMathOperator{\res}{Res}
\DeclareMathOperator{\cores}{Cores}

\newcommand{\ord}{\mathrm{ord}}

\newcommand{\cO}{\mathcal{O}}

\newcommand{\cH}{\mathcal{H}}
\newcommand{\cN}{\mathcal{N}}

\newcommand{\sH}{\mathscr{H}}

\newcommand{\cL}{\mathcal{L}}
\newcommand{\cK}{\mathcal{K}}
\newcommand{\cX}{\mathcal{X}}
\newcommand{\bz}{\mathbf{z}}

\newcommand{\BDP}{\mathrm{BDP}}
\newcommand{\BK}{\mathrm{BK}}
\renewcommand{\div}{\mathrm{div}}

\newcommand{\AJ}{\mathrm{AJ}}

\DeclareMathOperator{\chow}{CH}

\newcommand{\col}[2]{\left[\begin{matrix}
		#1\\ #2
	\end{matrix}\right]}

\newcommand{\mat}[4]{\left[\begin{matrix}
		#1 & #2\\
		#3 & #4\\
	\end{matrix}\right]}

\newcommand{\ann}{\mathrm{Ann}}
\newcommand{\barv}{{\bar{v}}}
\newcommand{\commentold}[1]{}
\newcommand{\loc}{\mathrm{loc}}

\newcommand{\fCol}{\mathfrak{Col}}
\newcommand{\Frob}{\mathrm{Frob}}
\newcommand{\fF}{\mathfrak{F}}
\newcommand{\fo}{\mathfrak{o}}
\newcommand{\bff}{\mathbf{f}}
\newcommand{\cris}{\mathrm{cris}}

\newcommand{\squarefree}{\mathscr{K}}
\newcommand{\cZ}{\mathcal{Z}}
\newcommand{\cM}{\mathcal{M}}

\title[Selmer groups over anticyclotomic towers]{On the structure of the Bloch--Kato Selmer groups of modular forms over anticyclotomic $\mathbf{Z}_p$-towers}

\author[A.~Lei]{Antonio Lei}
\address[Lei]{Department of Mathematics and Statistics\\University of Ottawa\\
150 Louis-Pasteur Pvt\\
Ottawa, ON\\
Canada K1N 6N5}
\email{antonio.lei@uottawa.ca}

\author[L.~Mastella]{Luca Mastella}
\address[Mastella]{Faculty of Mathematics and Computer Science\\
UniDistance Suisse\\Schinerstrasse 18\\Brig\\3900\\Switzerland}
\email{luca.mastella@unidistance.ch}

\author[L.~Zhao]{Luochen Zhao}
\address[Zhao]{Morningside Center of Mathematics\\No.55 Zhongguancun East Road\\Beijing\\100190\\China}
\email{luochenzhao@amss.ac.cn}


\usepackage[backend=biber, style=alphabetic, doi=false, url=false, isbn=false]{biblatex}
\begin{document}

\begin{abstract}
    Let $p$ be an odd prime number and let $K$ be an imaginary quadratic field in which $p$ is split. Let $f$ be a modular form with good reduction at $p$. We study the variation of the Bloch--Kato Selmer groups and the Bloch--Kato--Shafarevich--Tate groups of $f$ over the anticyclotomic $\mathbf{Z}_p$-extension $K_\infty$ of $K$. In particular, we show that under the generalized Heegner hypothesis, if the localization of the generalized Heegner cycle attached to $f$ at one of the primes above $p$ is primitive and certain local conditions hold, then the Pontryagin dual of the Selmer group of $f$ over $K_\infty$ is free over the Iwasawa algebra. Consequently, the Bloch--Kato--Shafarevich--Tate groups of $f$ vanish. This generalizes earlier works of Matar and Matar--Nekov\'a\v{r} on elliptic curves. Furthermore, our proof applies uniformly to the ordinary and non-ordinary settings.
\end{abstract}

\maketitle

\section{Introduction}

\subsection{Setting and notation}\label{sec:setting}

Fix once and for all a prime number $p > 2$, an algebraic closure $\Qbar$ of $\Q$ and $\Qpbar$ of $\Qp$, and an embedding $\iota_p \colon \Qbar \inj \Qpbar$.
Let $N \ge 1$ be an integer and $f = \sum_{n>0} a_n(f)q^n$ be a  normalized cuspidal eigen-newform of even weight $k \ge 2$ and level $\Gamma_0(N)$, and let $\fF = \Q_p\left(a_n(f):{n >0}\right)$ be the finite extension of $\Q_p$ generated by the Fourier coefficients of $f$ under the embedding $\iota_p$; denote by $\fo$ the ring of integers of $\fF$ and fix a uniformizer $\varpi$ of $\fo$.
Deligne \cite{deligne:formes-mod-reps-l-adiques} attached to $f$ a $p$-adic representation 
\[
    \rho_{f} \colon G_\Q \longrightarrow \Gl_2(\fF)
\]
that is unramified outside of $ pN$. Denote by $V_f$ the underlying vector space, and let $V= V_f(k/2)$ be its central critical twist. When $p \nmid 2(k-2)!N\varphi(N)$, where $\varphi(N)$ is Euler's totient function, Nekov\'a\v{r}  \cite[\S3]{nekovar92} constructed a lattice $T\subset V$ that is endowed with a $G_\Q$-equivariant skew-symmetric perfect pairing
\[
T \times T \to \fo(1),
\]
making $T$, and hence $V$, self-dual. That is, $T \simeq \Hom_{\fo}(T,\fo(1))$ and $V \simeq \Hom_{\fF}(V,\fF(1))$; \textit{cf.}, \cite[(14.10.1)]{kato:p-adic-hodge-theory-zeta-functions-mod-forms}. We put $A = V/T$. 

Let $K$ be an imaginary quadratic field of discriminant $d_K \ne -3,-4$ coprime to $Np$. Write $N$ as the product $N=N^+N^-$ with $N^+$ (resp.~$N^-$) divisible only by primes that split (resp. remain inert) in $K$. We assume that $(K,N)$ satisfies the \emph{generalized Heegner hypothesis}:
\begin{align}\tag{Heeg.}\label{ass:genHeeg}
    N^- \text{ is a squarefree product of an even number of primes}.
\end{align}

Next, let $K_\infty/K$ be the anticyclotomic $\Z_p$-extension of $K$, and $K_n$ be the intermediate extension with $\gal(K_n/K)\simeq \Z/p^n\Z$. We will assume $p\nmid h_K$, so that both places above $p$ in $K$ are totally ramified in $K_\infty/K$. By an abuse of notation, we will denote by $v$ and $\barv$ the two places above $p$ in $K_n$ for $n\in \Z_{\ge 0}\cup \{\infty\}$. Put $\Gamma = \Gal(K_\infty/K)$ and $\Lambda = \fo[[\Gamma]]$.

Finally, if $M$ is a $\Zp$-module, we denote by $M^\vee$ its Pontryagin dual $\Hom(M,\Q_p/\Z_p)$, and by $M_\div$ its maximal divisible subgroup.

\subsection{Background}
    In this paper, we are interested in the anticyclotomic Iwasawa theory of Selmer groups attached to the form $f$, assuming  local primitivity at $v$ of the associated generalized Heegner class. We shall begin with a brief historical account which will provide some context to our results to be stated below. We invite the intrigued reader to consult \cite[\S0]{matar-nekovar} for a much more detailed overview.
	
	Our study of the Selmer groups is deeply rooted in Kolyvagin's original breakthrough \cite{kolyvagin90}, where, under some assumptions, a bound of the Shafarevich--Tate group over $K$ is given when $f$ corresponds to a rational non-CM elliptic curve $E$ for which the Heegner point $y_K\in E(K)$ is nontorsion. Gross \cite{gross-durham} subsequently gave a self-contained proof in the simplified setting when $y_K\notin pE(K)$, in which case the $p$-part of the Shafarevich--Tate group of $E$ over $K$ is trivial, and hence $\Sel_{p^\infty}(E/K)$ coincides with $E(K)\otimes \Q_p/\Z_p$, generated by $y_K$.
 
 Assuming furthermore that $E$ has good ordinary reduction at $p$, the work of Matar--Nekov\'a\v{r} \cite{matar-nekovar}, among other things, put the result of Kolyvagin--Gross in the Iwasawa theoretic context, and proved the vanishing of the $p$-primary Shafarevich--Tate groups of $E$ in the anticyclotomic $\Zp$-tower. Furthermore, they studied the structure of $\Sel_{p^\infty}(E/K_n)$ and showed that the Pontryagin dual of $\Sel_{p^\infty}(E/K_\infty)$ is a free $\Lambda$-module of rank one. When $E$ has good supersingular reduction at $p$, building on the work of Longo--Vigni \cite{longo-vigni-plus-minus} and the plus/minus theory of Kobayashi \cite{kobayashi03}, Matar \cite{matar-supersingular,matar:plus-minus} extended the results of \cite{matar-nekovar} to show that the $p$-primary Shafarevich--Tate groups of $E$ in the anticyclotomic $\Zp$-tower are once again trivial, whereas the Pontryagin dual of $\Sel_{p^\infty}(E/K_\infty)$ is a free $\Lambda$-module of rank two.
	
	Recently, the second-named author \cite{mastella23} established the analogous result of Kolyvagin--Gross in the higher weight case, with the Heegner point replaced by the generalized Heegner class $z_{f,K}\in H^1(K,T)$ (the exact definition is recalled in \S\ref{sec:heegner}). In this vein, our goal here is to study the variation of Selmer groups of higher weight modular forms in the anticyclotomic tower.

\subsection{Main result}

We now state our main theorem, which corresponds to Proposition \ref{prop:bdp}, Corollary \ref{cor:bdp-sha-trivial} and Theorem \ref{thm:lambda-ranks} in the main body of the article. Below, for a number field $L$, $\Sel(L,A)$ denotes the usual Bloch--Kato Selmer group.
\begin{mainThm}\label{mainthm}
    Let $f$ be a normalized cuspidal eigen-newform of even weight $k$ and level $\Gamma_0(N)$. Let $K$ be an imaginary quadratic field with coprime-to-$Np$ discriminant $d_K\ne -3,-4$ and satisfying \eqref{ass:genHeeg}. Assume moreover $N\ge 5$ if $N^-=1$, and $N^+> 3$ and $k\ge4$ if $N^-\ne 1$. Let $p$ be a prime such that:
    \begin{enumerate}
		\item[(i)] $p$ splits in $K$, and does not divide the class number $h_K$;
		\item[(ii)] $p\nmid 2(k-2)!N\varphi(N)$;
		\item[(iii)] if $f$ has weight 2 and is $p$-ordinary, then $a_p\not\equiv 1\pmod \varpi$;
		\item[(iv)] for any $w\mid N$ in $K$ and any $w'\mid w$ in $K_\infty$, 
       the group $A^{G_{K_{\infty,w'}}}$ is divisible;
		\item[(v)] the class $z_{f, K}$ is locally primitive at $v$, i.e., its restriction $\loc_v(z_{f,K})\in H^1_\bff(K_v,T)$ is not in $\varpi H^1_\bff(K_v,T)$;
        \item[(vi)] $\Sel(K,A)= \fF/\fo\cdot z_{f,K}$.
	\end{enumerate}
	Then we have:
	\begin{enumerate}
		\item for $n\in \Z_{\ge 0}\cup \{\infty\}$, the relaxed-strict (or BDP) Selmer group $\Sel^{\emptyset,0}(K_n, A)$ is trivial;
		\item for $n\in \Z_{\ge 0}$, $\Sel(K_n,A)\simeq(\fF/\fo)^{\oplus p^n}$;
		\item the Pontryagin dual $\Sel(K_\infty,A)^\vee$ is a free $\Lambda$-module of rank one (resp.~two) if $f$ is ordinary (resp.~non-ordinary) at $p$.
	\end{enumerate}
\end{mainThm}
\begin{rem}
    We now explain the \textit{raisons d'\^{e}tre} of the hypotheses of Theorem~\ref{mainthm}. The splitting assumption (i) is inherent in our approach, namely the use of BDP Selmer groups in the study of usual Selmer groups; see \S\ref{subsec:strategy} for more details. As mentioned earlier, condition (ii) originates from Nekov\'a\v{r}'s construction of the lattice $T$. Entry (iii) gives the triviality of $A[\varpi]^{G_{K_{\infty,w}}}$ for $w\,|\,p$, which ensures certain desirable properties of the local and global cohomology groups under consideration (see Lemma~\ref{lem:torsion}). 
    Condition (iv) is a higher-weight variant of the $p$-freeness of Tamagawa numbers (\textit{cf.}~\cite[remark following Lemma 3.3]{greenberg-cetraro}), and is needed for the exact control theorem Proposition \ref{prop:control}. The premises (v)-(vi) are key inputs of our proof of the vanishing of the BDP Selmer groups. The local primitivity statement (v) is analogous to the assumption that the Heegner point $y_K$ of an elliptic curve $E$ satisfies $y_K\notin pE(K_v)$ in \cite[Theorem 1.4]{matar-supersingular}. Note that (v) implies $z_{f,K}\notin \varpi H^1(K,T)$, and thus if $f$ does not have complex multiplication in the sense of \cite[p.~34]{ribet:nebentypus} and $p$ is large enough, (vi) is true by a result of the second-named author (see Theorem \ref{th:mastella-th-0.2} and Remark \ref{rem:ribet}).
\end{rem}
\begin{rem}
    As the reader will notice, our proof of Theorem \ref{mainthm} in \S\ref{sec:bdp} and \S\ref{sec:universal-norms} does not employ the generalized Heegner class in an essential way. Rather, we rely on the existence of a cohomology class $z\in H^1_\bff(K,T)$ for which the conditions given by hypotheses (v) and (vi) above are met; note that any two such classes $z,z'$ differ only by an element of $\fo^\times$. As such, when $N^- = 1$, Theorem \ref{mainthm} holds when $z_{f,K}$ is replaced by the classical Heegner class $P(1)\in H^1_\bff(K,T)$ constructed by Nekov\'a\v{r} \cite{nekovar92}, where the corresponding statement of (vi) is a result of Besser \cite[proof of Theorem 1.2]{besser97} under some assumptions on $p$ (when $N^-\ne 1$, there is also a partial result in this direction given in \cite{elias-dvp}, where the bound of Shafarevich--Tate groups is not made explicit). In this article, we settle for using the class $z_{f,K}$, because
    \begin{itemize}
        \item condition (vi) is known to be valid by the work of the second-named author under the generalized Heegner hypothesis together with certain additional assumptions; see \S\ref{sec:heegner} for an extensive discussion of this result;
        \item from the perspective of the BDP Iwasawa main conjecture, the vanishing of $\Sel^{\emptyset,0}(K_\infty,A)$ is tied with $z_{f,K}$ being globally primitive in $H^1(K,T)$, as explained in Remark \ref{rem:BDP}. Therefore, it is natural to choose $z_{f,K}$ as our ``prima donna''.
    \end{itemize}
\end{rem}

\subsection{Strategy of proof}\label{subsec:strategy}
Our proof can be regarded as a synthesis of two important philosophies: that of Matar--Nekov\'a\v{r} \cite{matar-nekovar}, whose insight is that precise information of Selmer modules varying in the tower can be obtained from that at the bottom via purely Iwasawa-theoretical methods; and that of Kobayashi--Ota \cite{kobayashi-ota}, which emphasizes the utility of the BDP Selmer groups that can be handled uniformly for both ordinary and non-ordinary primes, making them remarkable surrogates to the usual Selmer groups whose behaviors are far more intricate.

More precisely, our proof of Theorem \ref{mainthm} proceeds as follows. First, in \S\ref{sec:bdp}, we establish the vanishing of BDP, i.e., relaxed-strict, Selmer groups, under the assumption of local primitivity of $z_{f,K}$ at the prime $v$; along the way, we review the control theorem for BDP Selmer groups in \S\ref{sec:control}. The rest of \S\ref{sec:bdp} then uses the vanishing to prove that the localizations $\Sel(K_n,A)\to H^1_\bff(K_{n,w},A)$ are isomorphisms for $n\in \Z_{\ge 0}$ and $w\mid p$ (the subscript $\bff$ denotes the Bloch--Kato local condition, and is recalled in \S\ref{sec:selmer}). On the one hand, this gives the growth formula for Selmer groups. On the other, it converts the study of Selmer groups to that of local cohomology groups, which we investigate in \S\ref{sec:universal-norms} using the theory of universal norms from $p$-adic Hodge theory.

\subsection{Obiter dictum}
As an aside, in Appendix \ref{app}, we present an alternative approach to results similar to Theorem \ref{mainthm} in the special case where $a_p(f) = 0$. This approach is modeled on that of Matar \cite{matar-supersingular}, which is based on the study of the (generalized) Heegner modules and crucially uses the plus/minus theory originally developed by Kobayashi in \cite{kobayashi03} (and later extended to modular forms in \cite{lei-compositio,lei10}). In this case, under some extra condition on the first and second generalized Heegner classes, the corank growth formula can be proved by assuming the global primitivity of $z_{f,K}$ (see Theorem \ref{thm:B}.(II)), which complements Theorem \ref{mainthm}. As part of the proof, we show that the Pontryagin duals of the plus and minus Selmer groups are free of rank one over $\Lambda$, generalizing \cite[Theorem~3.1]{matar-supersingular}. When $a_p(f)\ne0$, even though Sprung's $\sharp/\flat$ theory \cite{sprung} (and the higher weight generalization) is available in the anticyclotomic setting (see \cite{BL-21}), it is not clear to the authors how to extend the proof of Theorem~\ref{thm:B} to this generality. The main obstacle is the lack of an explicit description of the $\sharp/\flat$ Selmer groups over finite extensions $K_n$. Such description plays a crucial role in the proof of Theorem~\ref{thm:B}.

\subsection{Acknowledgement}
We thank Ming-Lun Hsieh for very helpful discussions regarding \cite{castella-hsieh:heegner-cycles-p-adic-l-functions} and Lemma \ref{lem:bad-reduction}, and Xin Wan for pointing out a gap in our argument. We are also grateful to Gautier Ponsinet for answering our questions about \cite{ponsinet20}. Thanks are also due to K\^{a}z{\i}m B\"uy\"ukboduk, Chan-Ho Kim, Naman Pratap, Ari Shnidman and Stefano Vigni for enlightening conversations and helpful comments. AL thanks his coauthors for inviting him to join this collaboration after it was initiated. He gives special thanks to LZ for spearheading the work on the appendix. LM thanks his coauthors for having pointed out many improvements and simplifications on his work. LZ is grateful to Meng Fai Lim for pointing out to him the paper \cite{mastella23} that led to this project, and to AL for continuously sharing valuable advice on becoming a better mathematician. He also thanks Xin Wan and the Morningside Center of Mathematics (MCM) for their generous support and warm hospitality during his visit.
Finally, the authors thank the anonymous referee for helpful comments and suggestions on an earlier version of the article, which led to many improvements.

 AL's research is supported by the NSERC Discovery Grants Program RGPIN-2020-04259 and RGPAS-2020-00096. LM is partially supported by PRIN 2022 The arithmetic of motives and $L$-functions. LZ is supported in part by the European Research Council (ERC, CurveArithmetic, 101078157), and his visit to MCM is supported by Xin Wan's National Key R\&D Program of China grant 2020YFA0712600 and NSFC grant 12288201. 

\section{Selmer groups}
\label{sec:selmer}

From now on, we will always assume that $p$ splits in $K$ and $p\nmid 2(k-2)!N\varphi(N)h_K$.

In this section, we introduce the various notions of Selmer groups that will be utilized throughout the paper. Let $S = \{\ell : \ell \mid N\}$ and for any finite extension $E/K$, let $S_E$ denote the set of primes of $E$ above those of $S$. For $n\in \Z_{\ge 0}\cup\{\infty\}$ we also use the shorthand $S_n$ in place of $S_{K_n}$. Write $\Sigma = S \cup \{p, \infty\} $, $\Sigma_E = S_E \cup \{w \mid p\infty\}$ and let $E_\Sigma$ be the maximal unramified extension of $E$ unramified outside $\Sigma_E$.

Recall that for any place $w$ of $E$ the Bloch--Kato finite local condition $\fcond$ on $V$ is defined as
\[
\honef(E_w,V) = 
    \begin{cases}
        \ker\big(H^1(E_w,V)\to H^1(E_w,V\otimes B_\cris)\big) &\text{ if }w\mid p;\\
        \ker\big(H^1(E_w,V)\to H^1(I_w,V)\big)&\text{ if }w\nmid p,
    \end{cases}
\]
where $I_w \subset G_{E_w}$ is the inertia subgroup at $w$ and $B_\cris$ is Fontaine's ring of cristalline periods. This defines also local subgroups $\honef(E_w, A) \subseteq \hone(E_w, A)$ and $\honef(E_w, T) \subseteq \hone(E_w, T)$ taking respectively its image and preimage via the morphisms induced by the natural projection $V \surj A$ and inclusion $T \subset V$.

For $M = T, V, A$ the Bloch--Kato Selmer group is defined as
\[
    \Sel(E, M) = \ker \Biggl(\hone(E_\Sigma, M) \to \prod_{\substack{w \in \Sigma_E\\w \nmid \infty}} \frac{\hone(E_w, M)}{\honef(E_w, M)} \Biggr).
\]
Note that the infinite places are ignored since $p>2$. 
For $E=K_n$ we will consider additional Selmer groups by varying the local conditions at places dividing $p$. Thus for $w\in \{v,\barv\}$ a place of $K_n$ we introduce:
\begin{itemize}
    \item the relaxed condition $\emptyset$ defined by
    \[
        H^1_\emptyset(K_{n,w},A) = H^1(K_{n,w},A);
    \]
    
    \item the strict condition $0$ defined by
    \[
        H^1_0(K_{n,w},A) = 0\subseteq H^1(K_{n,w},A).
    \]
\end{itemize}
We define for $\star, \bullet \in \{\fcond, 0, \emptyset\}$, the $(\star, \bullet)$-Selmer group of $A$ as 
\[
    \Sel^{\star, \bullet}(K_n, A) = \ker \Biggl(\hone(K_{n, \Sigma}, A) \to \prod_{w \in S_n} \frac{\hone(K_{n, w}, A)}{\honef(K_{n, w}, A)} \times 
    \frac{\hone(K_{n, v}, A)}{\hone_{\star}(K_{n, v}, A)} \times \frac{\hone(K_{n, \bar{v}}, A)}{\hone_{\bullet}(K_{n, \bar{v}}, A)}\Biggr).
\]
Note that $\Sel^{\fcond, \fcond}(K_n, A) = \Sel(K_n, A)$. In addition, we put $\Sel^{\star,\bullet}(K_\infty,A) = \displaystyle\ilim_n \Sel^{\star,\bullet}(K_n,A)$. We will call $\Sel^{\emptyset, 0}(K_n, A)$ the $\BDP$-Selmer group of $A$ and will denote it by $\Sel^\BDP(K_n, A)$, its name derived from its link with the
$p$-adic $L$-function of Bertolini--Darmon--Prasanna \cite{bdp:generalized} and Brako\v{c}evi\'{c} \cite{miljan} via the corresponding Iwasawa main conjecture; see \cite[Theorem~3.4]{castella17}, \cite[Theorem 1.5]{kobayashi-ota} and \cite[Theorem A]{lei-zhao}, as well as our discussion in Remark~\ref{rem:BDP} below.

Occasionally, we shall also consider compact Selmer groups over $K_n$:
\begin{align*}
    \Sel^{\star,\bullet}(K_n,T) = \plim_m\Sel^{\star,\bullet}(K_n,A)[p^m].
\end{align*}

\section{Review of generalized Heegner cycles}
\label{sec:heegner}

In what follows, we work with a distinguished set of cohomology classes $z_{f, c} \in \Sel(K[c], T)$, where $K[c]$ denotes the ring class field of $K$ of conductor $c\in \Z_{>0}$. 
Throughout this section, we shall assume $k>2$ when $N^-\ne1$, as this is the setting of \cite{magrone:generalized-heegner-cycles-quaternionic}, where such classes are constructed in the quaternionic case.

In the case $N^- = 1$, we define $z_{f, c} \in \hone(K[c], T)$ to be the class $z_{f, \chi, c}$ of \cite[(4.6)]{castella-hsieh:heegner-cycles-p-adic-l-functions} for $\chi$ the trivial character. These classes are the image of (a subclass of) the generalized Heegner cycles of \cite{bdp:generalized} via a suitable étale Abel--Jacobi map (see \cite[\S4.2, 4.4]{castella-hsieh:heegner-cycles-p-adic-l-functions} and \cite[\S3.3]{mastella23})
\[
\AJ_{K[c]} \colon \chow^{k-1}(X_{k-2}/K[c])_0 \otimes \fo \to \hone(K[c], T).
\]
Here,  $X_{k-2}$ is the so called \emph{generalized Kuga--Sato variety}, namely the product over the modular curve $X_1(N)$ of the Kuga--Sato variety $W_{k-2}$ and $(k-2)$-copies of a CM elliptic curve $A$ defined over $K[1]$. As explained in \cite[Remark 4.8]{castella-hsieh:heegner-cycles-p-adic-l-functions}, 
the image of $\AJ_{K[c]}$ is contained in $\Sel(K[c], T)$, so in particular $z_{f, c} \in \Sel(K[c], T)$.

When $N^- \ne 1$, we can define such classes in a similar way, after replacing $W_{k-2}$ with the Kuga--Sato variety defined over a (quaternionic) Shimura curve and $A$ with a false elliptic curve with CM by $\cO_K$. This construction is carried out in \cite{brooks:shimura-curves-l-fcts} and \cite{magrone:generalized-heegner-cycles-quaternionic}.
We write $z_{f, c}$ for the class $z_{\chi, c}$ of \cite[(5.2)]{magrone:generalized-heegner-cycles-quaternionic} with $\chi$ taken to be the trivial character. We shall abuse the notation and use the same letters $\AJ_{K[c]}$ to denote the étale Abel--Jacobi map of \cite[\S5.8-5.9]{magrone:generalized-heegner-cycles-quaternionic}; again its image is contained in $\Sel(K[c], T)$ and in particular $z_{f, c} \in \Sel(K[c], T)$ as noted in \cite[proof of Proposition 7.9]{magrone:generalized-heegner-cycles-quaternionic}.

In both cases, we let
\[
z_{f, K} = \cores_{K[1]/K}(z_{f, 1}) \in \Sel(K, T).
\]
The classes $z_{f,c}$ satisfy certain norm relations as $c$ varies, forming an (anticyclotomic) Euler system  
(see \cite[Proposition 7.4]{castella-hsieh:heegner-cycles-p-adic-l-functions} for the case $N^- =1$, \cite[Proposition 7.9]{magrone:generalized-heegner-cycles-quaternionic} for the case $N^-\ne 1$). Let $\squarefree$ be the set of square-free integers $n$ that are products of primes $\ell \nmid 2Np$ that are inert in $K$. For any $n = m \ell \in \squarefree$, with $\ell$ a prime, we have
\begin{enumerate}
    \item[(E1)] $\cores_{K[n]/K[m]}(z_{f, n}) = a_\ell z_{f, m}$; 
    \item[(E2)] for any compatible choice of primes $\lambda_n$ of $K[n]$ and $\lambda_m$ of $K[m]$ lying above $\ell$, we have $\loc_{\lambda_n}(z_{f, n})= \res_{K[m]_{\lambda_m}/K[n]_{\lambda_n}}(\Frob_\ell \cdot \loc_{\lambda_m}(z_{f, m}))$. Here, $\loc_{\lambda_s} \colon \hone(K[s], T) \to \hone(K[s]_{\lambda_s}, T)$ denotes the localization at $\lambda_s$ for $s = n, m$, and $\Frob_\ell \in \Gal(\Q_\ell^\ur/\Q_\ell)$ is the Frobenius element at $\ell$;
    \item[(E3)] $\tau \cdot z_{f, n} = w_f (\sigma \cdot z_{f, n})$, where $\tau$ denotes the complex conjugation, $w_f \in \{\pm 1\}$ is the eigenvalue of $f$ with respect to the Atkin--Lehner involution and $\sigma \in \Gal(K[n]/K)$ is the image of $\bar{\mathfrak{N}}$ via the Artin map.
\end{enumerate}

Using these properties, Castella--Hsieh and Magrone (see \cite[Theorem~7.7]{castella-hsieh:heegner-cycles-p-adic-l-functions} and \cite[Theorem 7.11]{magrone:generalized-heegner-cycles-quaternionic}) showed that if $z_{f, K}$ is not $\fo$-torsion, then $\Sel(K, V) = \fF z_{f,K}$. More precisely,  they prove that there exists a constant $C >0$ such that $p^C$ annihilates the quotient $\Sel(K, A)/(\fF/\fo)z_{f, K}$ (see in particular~\cite[Theorem 7.19]{castella-hsieh:heegner-cycles-p-adic-l-functions}). The second-named author of the present article improved this result in \cite[Theorem~0.2]{mastella23} by computing the constant $C$ in the case $N^-=1$, $k>2$, $p\ne3$ and $f$ is $p$-ordinary together with certain assumptions on $p$. It was shown that if $z_{f,K}$ is not $p$-divisible, then $C=0$, or equivalently $\Sel(K, A) = (\fF/\fo)z_{f, K}$ (cf.~\emph{op.~cit.}~Theorem 0.1). The following theorem is a slight extension of the aforementioned result. Note that we drop the hypotheses of $p$-ordinarity and the condition that $k>2$ since they are not used in the proof therein. Moreover, the exclusion of $p=3$ ensured the validity of Lemma 4.5 \textit{ibid.} (as explained in Remark 4.4 \textit{ibid.}), for which we will provide a different proof below (see Lemma~\ref{lemma:restriction-iso}). Furthermore, we relax the hypothesis $N^- =1$ using the quaternionic generalized Heegner cycles of \cite{magrone:generalized-heegner-cycles-quaternionic}.

\begin{thm}\label{th:mastella-th-0.2}
Let $p$ be a prime such that 
\begin{itemize}
    \item $p \nmid 2N \varphi(N)(k-2)!$;
    \item $p$ is unramified in $\fF$;
    \item $\rho_{f}$ has \emph{big image}, i.e., its image contains the set of matrices
    \[
    \left\{g \in \Gl_2(\fo) : \det g \in (\Z_p^\times)^{k-1}\right\}.
    \]
\end{itemize}
If $z_{f,K}$ is not $\fo$-torsion and is not divisible by $p$ as an element of $\Sel(K, T)$, then
\[
\Sel(K, A) = (\fF/\fo)z_{f, K}.
\]
\end{thm}

\begin{rem}\label{rem:ribet}
Our assumption on the big image implies that $f$ is non-CM for the following reason. As discussed in \cite[first paragraph of p.~186]{ribet:rankin}, if $f$ were a CM form, $\mathrm{Lie}(\rho_f)$ could be identified with $L\otimes \Q_p$, where $L$ is the field of  ``complex multiplication'' of $f$ \cite[p.~34, Remark~1]{ribet:nebentypus}. However, the big image condition implies that this Lie algebra should contain $\mathfrak{sl}_2(\fF)$, which is a contradiction.
Consequently, \cite[discussion at the end of \S3]{ribet:rankin} tells us that Theorem~\ref{th:mastella-th-0.2} applies to all but finitely many $p$.
\end{rem}

\begin{proof}[Proof of Theorem \ref{th:mastella-th-0.2}]
The proof follows the same line of the argument of \cite[Theorem~0.2]{mastella23}. We give a brief summary here for the convenience of the reader and explain why the condition $N^-=1$ and the $p$-ordinarity of $f$ are not necessary. Let $M\ge1$ be an integer and write $\squarefree_M$ for the set of square-free products of prime numbers $\ell$ such that 
\begin{enumerate}
    \item $\ell \nmid Npd_K$;
    \item $\ell$ inert in $K$;
    \item $p^M \mid a_\ell, \ell+1$;
    \item $p^{M+1} \nmid \ell + 1 \pm a_\ell$.
\end{enumerate}
Applying the Kolyvagin derivative operator to the classes $z_{f, n}$, we may construct as in \cite[\S4.2]{mastella23} a family of classes $P(n) \in \hone(K, A[p^M])$, $n\in\squarefree_M$ and show that the properties (E1)-(E3) and the inclusion $z_{f, n} \in \Sel(K[n], T)$ imply that \cite[Proposition~3.2]{besser97} applies to our current setting. In other words, for any $n\in\squarefree_M$, we have:
\begin{enumerate}
    \item $P(n)$ belongs to the $\epsilon_n$-eigenspace of the complex conjugation acting on $\hone(K, A[p^M])$, where $\epsilon_n = (-1)^{\omega(n)}w_f \in \set{\pm1}$ (here, $\omega(n)$ denotes the number of prime factors of $n$);
    \item for any $v \nmid Nn$, $\loc_v(P(n)) \in \honef(K_v, A[p^M])$;
    \item if $n = m \cdot \ell$, there is an isomorphism $\honef(K_\lambda, A[p^M]) \cong \hones(K_\lambda, A[p^M])$, where $\hones = \hone/\honef$  is the singular quotient, such that $\loc_\lambda P(m)$ corresponds to $[\loc_\lambda P(n)]_s$ (the image of $\loc_\lambda P(n)$ in $\hones(K_\lambda, A[p^M])$). 
\end{enumerate}
The fact that the classes $P(n)$ satisfy properties (1)-(3) and our assumptions on $p$ are the input of \cite[\S6]{besser97} (where again $f$ is not assumed to be ordinary at $p$). We can therefore apply verbatim the results therein. In particular, the quotient 
\[
\Sel(K, A)/(\fF/\fo)z_{f, K} 
\]
is annihilated by $p^{2\mathcal{I}_p}$, where $\mathcal{I}_p$ denotes the smallest non-negative integer such that $z_{f, K}$ is non-zero in $\hone(K, A[p^{\mathcal{I}_p}])$. In particular, if $z_{f, K}$ is not divisible by $p$, we have  
\[
\Sel(K, A) = (\fF/\fo)z_{f, K}.\qedhere
\]
\end{proof}
We discuss consequences of Theorem~\ref{th:mastella-th-0.2}
on Shafarevich--Tate groups.
\begin{de}
    Let $E/K$ be a finite extension of fields. We define the Bloch--Kato--Shafarevich--Tate group of $f$ over $E$ as 
    \[
    \Sha_\BK(E, A) = \frac{\Sel(E, A)}{\Sel(E, A)_\div}.
    \]
\end{de}

Note in particular that, for any finite extension $E/K$, the group $\Sha_\BK(E, A)$ is finite since the Pontryagin dual $\Sel(E,A)^\vee$ is a finitely generated $\Zp$-module (see \cite[Proposition~B.2.7]{rubin:euler}).  In what follows, we give the definition of another type of Shafarevich--Tate groups via the Abel--Jacobi map, similar to the ones considered in \cite{nekovar92}. We first state the following technical lemma.

\begin{lem}\label{lemma:restriction-iso}
The restriction maps
\begin{gather*}
    \res_{K[1]/K} \colon \Sel(K, T) \to \Sel(K[1], T)^{\Gal(K[1]/K)},\\
    \res_{K[p^{n+1}]/K_n} \colon \Sel(K_n, T) \to \Sel(K[p^{n+1}], T)^{\Gal(K[p^{n+1}]/K_n)}, \ n\ge1
\end{gather*}
are isomorphisms.
\end{lem}

\begin{proof}
We follow a similar argument given in \cite[Lemma~5.5.9]{pratap}.   Let $(L,F)$ be either $(K[1],K)$ or $(K[p^{n+1}],K_n)$ and write $G=\Gal(L/F)$. Under the assumption $p\nmid h_K$, we have $p\nmid \#G$. 
    consider the following commutative diagram with exact rows:
\[\begin{tikzcd}
	0 & {\text{Sel}(F, A)} && {H^1(F_\Sigma,A)} &&  \prod_{v}\hones(F_v, A)  \\
	\\
	0 & {\text{Sel}(L, A)^G} && {H^1(L_\Sigma, A)^G} &&\prod_{w}\hones(L_w, A)^G
	\arrow[from=1-1, to=1-2]
	\arrow[from=1-2, to=1-4]
	\arrow[from=1-2, to=3-2]
	\arrow[from=1-4, to=1-6]
	\arrow["\res", from=1-4, to=3-4]
	\arrow["\prod_v g_v", from=1-6, to=3-6]
	\arrow[from=3-1, to=3-2]
	\arrow[from=3-2, to=3-4]
	\arrow[from=3-4, to=3-6]
\end{tikzcd}\]
    By the inflation-restriction sequence, the kernel and cokernel  of $\res$ are given by
    $ H^i(G, H^0(L, A))$, for $i=1,2$, respectively. As $H^0(L, A)$ is a pro-$p$ group while $p\nmid \#G$, \cite[Corollary 6.1]{atiyah-wall} tells us that $\res$ is an isomorphism. Hence, by the snake lemma, it suffices to show that $g_w$ is injective.

Let $v$ be a place of $F$ and let $w$ be a place of $L$ above $v$. Let $D=\Gal(L_w/F_v)$. We have the following commutative diagram where the rows are exact: \[\begin{tikzcd}[ampersand replacement=\&]
	{0 } \& {\honef(F_v,A)} \& {H^1(F_v,A)} \& {\hones(F_v,A)} \& 0 \\
	0 \& {\honef(L_w,A)^D} \& {H^1(L_w,A)^D} \& {\hones(L_w,A)^D.}
	\arrow[from=1-1, to=1-2]
	\arrow[from=1-2, to=1-3]
	\arrow["h_v",from=1-2, to=2-2]
	\arrow[from=1-3, to=1-4]
	\arrow["r_{A,v}",from=1-3, to=2-3]
	\arrow[from=1-4, to=1-5]
	\arrow["g_v",from=1-4, to=2-4]
	\arrow[from=2-1, to=2-2]
	\arrow[from=2-2, to=2-3]
	\arrow[from=2-3, to=2-4]
\end{tikzcd}\]
As the order of $D$ is coprime to $p$, the map $r_{A,v}$ is an isomorphism, as before. Thus, by the snake lemma again, it suffices to show that $h_v$ is surjective.

Suppose that $v|p$ and consider the following commutative diagrams
    \begin{equation}
        \begin{tikzcd}[ampersand replacement=\&]
	 {H^1(F_v,V)} \& {H^1(L_w,V)^D}\\
{H^1(F_v,V\otimes \Bcrys)} \& {H^1(L_w,V\otimes\Bcrys)^D,}
	\arrow["r_V",from=1-1, to=1-2]
	\arrow["b",from=1-1, to=2-1]
	\arrow["r_B",from=2-1, to=2-2]
	\arrow["b'",from=1-2, to=2-2]
	\end{tikzcd}\label{eq:diagram1}
    \end{equation}
    \begin{equation}
        \begin{tikzcd}[ampersand replacement=\&]
	 {H^1(F_v,V)} \& {H^1(L_w,V)^D}\\
{H^1(F_v,A)} \& {H^1(L_w,A)^D,}
	\arrow["r_V",from=1-1, to=1-2]
	\arrow["\pi",from=1-1, to=2-1]
	\arrow["r_A",from=2-1, to=2-2]
	\arrow["\pi'",from=1-2, to=2-2]
	\end{tikzcd}\label{eq:diagram2}
    \end{equation}
    where we have written $r_A$ for $r_{A,v}$ to simplify the notation.
    Note that $r_V$ and $r_B$ are all isomorphisms as $p\nmid \#D$.

Recall that $\honef(L_w,A)$ is defined as the image of $\honef(L_w, V)$ under the map induced by the natural map $V \to A$. This gives a surjection $\honef(L_w,V)\to\honef(L_w,A)$. As the kernel of this map is pro-$p$, the long exact sequence with respect to the group cohomology of $D$ gives a surjection  $\pi'_{\bff}:\honef(L_w,V)^D\to\honef(L_w,A)^D$.

Let $x_0\in \honef(L_w,A)^D$. Since $h_v$ is the restriction of $r_A$ to $\honef(F_v, A)$, the bijectivity of $r_A$ tells us that there exists $x\in H^1(F_v,A)$ such that $r_A(x)=x_0$.
In order to show the surjectivity of $h_v$, we show that $x$ is, in fact, an element of $\honef(F_v,A)$. This is equivalent to finding an element $z\in\honef(F_v,V)$ such that $\pi(z)=x$.

By the surjectivity of $\pi'_{\bff}$, there exists $y\in \honef(L_w,V)^D$ such that $\pi'(y)=x_0=r_A(x)$. In particular, $b'(y)=0$ by the definition of $\honef(L_w,V)$. Since $r_V$ is a bijection, we can find $z\in H^1(F_v,V)$ such that $r_V(z)=y$. Then $r_B\circ b(z)=0$ from the diagram \eqref{eq:diagram1}. The injectivity of $r_B$ implies that $b(z)=0$. So, $z\in \honef(K_v,V)$. From \eqref{eq:diagram2}, $r_A(x)=r_A\circ\pi(z)$. As $r_A$ is bijective, we have $x=\pi(z)$, as desired.

When $v \nmid p$, we have analogous commutative diagrams as \eqref{eq:diagram1} and \eqref{eq:diagram2}, with the map $r_B$ replaced by $r_I: H^1(I_v, V) \to H^1(I_w, V)^{I_w/I_v}$. Thus, the same conclusion holds, as desired.
\end{proof}

\begin{de}\label{def:AJ-image}
    We define the `modified' image of the Abel--Jacobi map 
    \[
    \Psi(K) = \res^{-1}_{K[1]/K}\Big( (\im \, \AJ_{K[1]})^{\Gal(K[1]/K)} \Big) \subseteq \Sel(K, T)
    \] and for any $n \ge 1$,  
    \[
    \Psi(K_n) = \res^{-1}_{K[p^{n+1}]/K_n}\Big( (\im \, \AJ_{K[p^{n+1}]})^{\Gal(K[p^{n+1}]/K_n)} \Big)\subseteq \Sel(K_n, T).
    \]
    The Abel--Jacobi--Shafarevich--Tate group of $f$ over $K_n$, for $n \ge 0$, is defined as
    \[
    \Sha_\AJ(K_n, A) = \frac{\Sel(K_n, A)}{\Psi(K_n) \otimes_{\fo} \fF/\fo}.
    \]
\end{de}

The reader is referred to \cite[Definitions~3.7 and~5.19]{mastella23}, where $\Psi$ is denoted by $\widetilde{\Lambda}_{\mathfrak{p}}$ and $\Sha_\AJ$ by $\widetilde{\Sha}_{\mathfrak{p}^\infty}$, for further details. {As remarked above, the definition of $\Sha_\AJ(K_n, A) $ is inspired by \cite{nekovar92}. Here, the modified image $\Psi(K_n)$ is introduced to ensure that  $z_{f, K} \in \Psi(K)$ and  $\alpha_n := \cores_{K[p^{n+1}]/K_n}(z_{f, p^{n+1}}) \in \Psi(K_n)$, for any $n\ge 0$. Since classical Heegner cycles are known to lie inside the image Abel--Jacobi map, no modification of the image was necessary in \emph{loc.~cit.} This inclusion facilitates the use of generalized Heegner cycles in the study of Shafarevich--Tate groups, which we will do in the appendix of the article.}

Note that $\Psi(K_n) \otimes \fF/\fo \subseteq \Sel(K_n, A)_\div$, by the commutativity of the following diagram for $n \ge 0$ 
\[
\begin{tikzcd}
    \Psi(K_n) \otimes \fF \ar[r]\ar[d, twoheadrightarrow] & \Sel(K_n, V) \ar[d]\\
    \Psi(K_n) \otimes (\fF/\fo) \ar[r] & \Sel(K_n, A) .
\end{tikzcd}
\]
Hence, there is a projection
\[
\Sha_\AJ(K_n, A) \surj \Sha_\BK(K_n, A).
\] 
As far as the authors are aware, it is not known whether this map is an isomorphism in general. However, under an additional assumption on $z_{f,K}$, we may prove the following result, which is analogous to \cite[Proposition~4.20]{longo-vigni:tamagawa} (see also Remark \ref{rem:serendipity}).
\begin{prop}
    Suppose that $z_{f, K}$ is not $\fo$-torsion. Then, $\Sha_\AJ(K, A) = \Sha_\BK(K, A)$.
\end{prop}

\begin{proof}
   When $N^-=1$, \cite[Theorem 7.19]{castella-hsieh:heegner-cycles-p-adic-l-functions} tells us that $\Sel(K,A)/(\fF/\fo)z_{f,K}$ is annihilated by a power of $p$ and thus finite. In particular, it follows that $\Sel(K,A)$ has corank $1$ over $\fo$. Since $z_{f, K} \in \Psi(K)$, the group $\Sha_\AJ(K, A)$ is a quotient of $\Sel(K,A)/(\fF/\fo)z_{f,K}$ and therefore again finite. Consequently, $\Psi(K)\otimes_\fo \fF/\fo$ and $\Sel(K, A)$ both have corank $1$ over $\fo$. The same can be said when $N^-> 1$ since the results of \cite[\S7.5]{castella-hsieh:heegner-cycles-p-adic-l-functions} can be readily generalized after replacing the anticyclotomic Euler system in \textit{loc.~cit.}~by the one constructed in \cite{magrone:generalized-heegner-cycles-quaternionic}.
    {Therefore, in both cases, we deduce that the $\fo$-modules $\Psi(K)\otimes_\fo \fF/\fo$ and $\Sel(K, A)_\div$ are cofree of corank one. In particular, they must be equal to each other. As $\Sha_\AJ(K, A) =\Sel(K,A)/ (\Psi(K)\otimes_\fo \fF/\fo)$ and $\Sha_\BK(K, A)=\Sel(K,A)/\Sel(K, A)_\div$,} the proposition follows.
\end{proof}

    From Theorem~\ref{th:mastella-th-0.2} we deduce that
    \begin{cor}\label{cor:mastella-th-0.2}
    Let $p$ be as in the statement of Theorem~\ref{th:mastella-th-0.2}. Suppose that $z_{f,K}$ is not $\fo$-torsion and that it is not divisible by $p$ as an element of $\Sel(K, T)$. Then,
    \[
    \Sha_\AJ(K, A) = \Sha_\BK(K, A) = 0.
    \]
    \end{cor}

    \begin{rem}
    \label{rem:serendipity}
    If $f$ is $p$-ordinary, Theorem \ref{th:mastella-th-0.2} can be used to prove (see \cite[Theorem 5.23]{mastella23}) that 
    \[
    \Sha_\AJ(K_\infty, A) :=  \frac{\Sel(K_\infty, A)}{\Big(\ilim_n \Psi(K_n)\Big) \otimes \fF/\fo} = 0.
    \] 
    In the current article, we prove that the same vanishing holds for the Bloch--Kato--Shafarevich--Tate group over $K_n$ for all $n\in \Z_{\ge 0}\cup\{\infty\}$ under suitable assumptions for both ordinary and non-ordinary primes. 
        Moreover, in the appendix, we extend the $\Sha_{\AJ}$ vanishing results of \cite{mastella23} to non-ordinary primes assuming  $a_p = 0$ (see Theorem~\ref{thm:B}.(V)). It should become apparent to the reader that the proofs we present here utilize techniques of a very different nature.

        In general, it is not known whether $\Sha_\AJ(K_n, A)$ and $\Sha_\BK(K_n, A)$ coincide. 
     In the ordinary setting, Theorem~\ref{mainthm} together with \cite[Theorem 5.23]{mastella23} imply that $\Sha_\AJ(K_\infty, A)=\Sha_\BK(K_\infty, A)=0$ under appropriate hypotheses. In the non-ordinary setting, when $a_p=0$, Theorem~\ref{thm:B}.(III) shows that $\Sha_\AJ(K_n, A)$ and $\Sha_\BK(K_n, A)$ coincide under the hypotheses therein (and vanish simultaneously under an additional condition). Such agreements obtained in this article seem to be a remarkable serendipity.
    \end{rem}

\section{Control theorem}
\label{sec:control}

We first introduce some notation. Recall that $K_n$ is the $n$-th layer of the anticyclotomic $\Z_p$-extension $K_\infty/K$, and we denote $\Gamma_n = \gal(K_\infty/K_n) = \Gamma^{p^n}$. Given $w\nmid p$ a place of $K$, let $w'\mid w$ be a place of $K_\infty$, and write $G_{K_\infty,w'}$ for the decomposition group of $w'$ in $G_{K_\infty}$. Define
\begin{align*}
    B_w = A^{G_{K_{\infty,w'}}}/\big(A^{G_{K_{\infty,w'}}}\big)_{\div};
\end{align*}
note that $b_w = |B_w|$ depends on $w$ and not on $w'$. Also, {since $B_w$ is a $p$-group, $p\nmid b_w$ if and only if $B_w=0$, namely $A^{G_{K_{\infty,w'}}}$ is divisible.} In what follows, for a Selmer condition $\star = \bff,0,\emptyset$, we will write $H^1_{/\star}$ for $H^1/H^1_\star$ for local and global cohomology groups.

We will work under the following hypotheses:
\begin{align}\tag{Tors.}\label{ass:tors}
    \text{if $f$ has weight 2 and is $p$-ordinary, then }a_p\not\equiv 1\pmod \varpi;
\end{align}
\begin{align}\tag{Tama.}\label{ass:tama}
    \text{for any place $w\mid N$ of $K$, }p\nmid b_w.
\end{align}

The assumption \eqref{ass:tors} should be thought of as a restraint on the torsion group $A$, as the following lemma shows. We note that in what follows, when referring to \eqref{ass:tors}, we are often having Lemma \ref{lem:torsion}.(a) in mind.

\begin{lem}\label{lem:torsion}
    Suppose \eqref{ass:tors} holds. Let $w$ be a place of $K$ lying above $p$. Then 
    \begin{itemize}
        \item[(a)] $H^0(K_{\infty,w},A)=0$;
        \item[(b)] suppose that $f$ is ordinary at $p$ and let $T^+$ denote the unique rank-one sub-representation of $T|_{G_{K_w}}$ whose Hodge--Tate weight is $k/2$. Let $A^-=\Hom(T^+,\fF/\fo)(1)$. Then, we have $H^0(K_{\infty,w},A^-)=0$.
    \end{itemize}
\end{lem}
\begin{proof}
    If $f$ is non-ordinary at $p$, the assertion in (a) is a consequence of a result of Fontaine (see, e.g., Lemma \ref{lem:lem4.4}). In the remainder of the proof, we assume that $f$ is ordinary at $p$. The representation $T|_{G_{K_w}}$ is of the form
    \[
    \begin{bmatrix}
        \delta\chi_\cyc^{k/2}&*\\0&\delta^{-1}\chi_\cyc^{1-k/2}
    \end{bmatrix},
    \]
    where $\chi_\cyc$ is the cyclotomic character, and $\delta$ is an unramified character that sends the geometric Frobenius to the unit root of $X^2-a_p(f)X+p^{k-1}$ (see \cite[p.~445]{longo-vigni-ny}). In particular, $$A^-\simeq\left(\fF/\fo\right)\left(\delta^{-1}\chi_\cyc^{1-k/2}\right).$$
    
    Let $I_w$ be the inertia group of $G_{K_w}$. Then the image of $I_w$ under $\chi_\cyc$ is $\Z_p^\times$. If $k\ge 4$, we have $-(p-1)<1-k/2< 0$ (thanks to our running hypothesis that $p\nmid 2(k-2)!N\varphi(N)$). Thus, $\delta^{-1}\chi_\cyc^{1-k/2}(I_w)\not\subseteq 1+p\Zp$. This implies that $H^0(I_w,A^-)=0$. If $k=2$, the hypothesis that $a_p(f)\not\equiv1\mod\varpi$ ensures that $\delta^{-1}\chi_\cyc^{1-k/2}=\delta^{-1}$ is not the trivial character modulo $\varpi$, so $H^0(K_w,A^-[\varpi])=0$ still. In both cases, we have $H^0(K_w,A^-)=0$. Consequently, we can apply an orbit-stabilizer argument, similar to the one given in \cite[proof of Lemma 5.2]{dionray} to deduce that $H^0(K_{n,w},A^-)=0$ for all $n\in \Z_{\ge 0}$: Suppose for the sake of contradiction that there is some nonzero $x\in H^0(K_{n,w},A^-)$. Let $M_x = \fo[\gal(K_{n,w}/K_w)]\cdot x\subseteq H^0(K_{n,w},A^-)$. Since $x$ is torsion over $\Zp$, the module $M_x$ is finite. Consider the partition of $M_x$ into the orbits under the $\gal(K_{n,w}/K_w)$-action. We have
	\begin{align}
		|M_x| = |\gal(K_{n,w}/K_w)\cdot 0| + \sum_{x'\ne 0} |\gal(K_{n,w}/K_w)\cdot x'|,
	\end{align}
	where the sum runs over the orbits in $M_x$ generated by  nonzero elements. Note that if $x'\ne 0$, then the orbit containing $ x'$ contains more than one element, as otherwise $x'\in H^0(K_{w},A^-)=0$, a contradiction. In particular, as $\gal(K_{n,w}/K_w)$ is a $p$-group, $|\gal(K_{n,w}/K_w)\cdot x'|$ is divisible by $p$. It follows that
    \begin{align*}
        |M_x|\equiv |\gal(K_{n,w}/K_w)\cdot 0| = |\{0\}| = 1\pmod p,
    \end{align*}
    which is impossible since $M_x$ is a nontrivial $p$-group. This shows that $H^0(K_{n,w},A^-)$ is identically zero. Now, recall that the Galois action on $V$ is continuous \cite[Proposition 3.1.(1)]{nekovar92}, so we have a continuous map $\mathcal{J}\colon G_{K_w}\times V\to V$ given by the group action. Hence, if $x\in V$, then $\mathcal{J}^{-1}(x+T)$ must contain $U\times (x+T)$ for some open subset $U$ of $G_{K_w}$. Thus any element of $A = V/T$ has a stabilizer that contains an open subset of $G_{K_w}$, whereby it must be open itself. As such, we find that any element in $H^0(K_{\infty,w},A^-)$ is stabilized by an open subgroup containing $G_{K_{\infty,w}}$, meaning it is invariant under $G_{K_{n,w}}$ for some $n$. In sum, $H^0(K_{\infty,w},A^-) = \cup_{n\ge 0} H^0(K_{n,w},A^-) = 0$, which proves part (b) of the lemma.

    We now turn our attention to part (a). As above, it suffices to show that $H^0(K_w,A)=0$. Let $\{e_1,e_2\}$ be a basis of $\fo^{\oplus 2}$ on which $G_{K_w}$ acts through the matrix $   \begin{bmatrix}
        \delta\chi_\cyc^{k/2}&\Xi\\0&\delta^{-1}\chi_\cyc^{1-k/2}
    \end{bmatrix}$. 
    Suppose $H^0(K_w,A)\ne0$. Then there exists an element $xe_1+ye_2\in \fo^{\oplus 2}$ such that, for all $\sigma\in G_w$, $$\sigma\cdot (xe_1+ye_2)\equiv xe_1+ye_2\pmod\varpi.$$
    In other words, for all $\sigma\in G_w$,
    $$\left(x\delta\chi_\cyc^{k/2}(\sigma)+y\Xi(\sigma)\right)e_1+y\delta^{-1}\chi_\cyc^{1-k/2}(\sigma)e_2\equiv xe_1+ye_2\pmod\varpi.$$
    In particular, for all $\sigma\in G_w$, $$y\delta^{-1}\chi_\cyc^{1-k/2}(\sigma)e_2\equiv ye_2\pmod\varpi.$$
    Our proof for part (b) tells us that we must have $y\equiv 0\mod\varpi$. This in turn implies that
    $$x\delta\chi_\cyc^{k/2}(\sigma)\equiv xe_1\pmod\varpi\ \text{for all }\sigma\in G_{K_w}.$$
    As $0<k/2<(p-1)$, we have once again $\delta\chi_\cyc^{k/2}(I_w)\not\subset 1+p\Zp$. Therefore, $H^0(I_w,\fF/\fo(\delta\chi_\cyc^{k/2}))=0$, which implies that $x\equiv 0\mod\varpi$. Hence, we conclude that $H^0(K_w,A)=0$, which implies the assertion of part (a).
\end{proof}

\begin{prop}\label{prop:control}
    Suppose $A$ satisfies assumptions \eqref{ass:tors} and \eqref{ass:tama}. Let $\star,\bullet \in \{0,\emptyset\}$. The restriction map on cohomology groups induces an isomorphism $\Phi_n:\Sel^{\star,\bullet}(K_n,A) \to \Sel^{\star,\bullet}(K_\infty,A)^{\Gamma_n}$ for all $n\in \Z_{\ge 0}$.
\end{prop}
While control theorems of this flavor are standard (see, e.g., \cite[Lemma 4.3]{matar-supersingular}, \cite[Theorem 4.1]{lei-lim-muller23} and \cite[Lemma 2.3]{ponsinet20}), we give a proof here for the reader's convenience. We begin with two lemmas.
\begin{lem}\label{lem:bad-reduction}
    Let $w\mid N$ be a place of $K$ and $F$ be an unramified finite extension of $K_w$. Then $A^{G_F}$ is finite.
\end{lem}
\begin{proof}
    For the sake of contradiction, suppose instead that there exist infinitely many $n\in \Z_{\ge 0}$ such that $x_n\in A^{G_F}$ and the annihilator ideal of $x_n$ is $(\varpi^n)\subseteq \fo$; we extend $x_n$ to all of $n\in \Z_{\ge 0}$ by setting $x_{n-1} = \varpi x_n$ if $x_n$ is defined and $x_{n-1}$ is not. Let $e$ be the ramification index of $\fF/\Q_p$ and put $y_n = x_{en}$. Denote by $(-,-)$ the pairing from \cite[Proposition 3.1.(2)]{nekovar92} on $T\times T$ and also $A[p^n]\times A[p^n]$ for $n\in \Z_{\ge 0}$ (note that since $T$ is self-dual, $A[p^n]\simeq (T/p^n)^\vee(1)$). Then we have $(p^{m-n}y_m,y_n) \in \mu_{p^n}$ for all $n,m\in \Z_{\ge 0}, m>n$.
		
		For $m,n$ as above, denote by $D_{m,n}$ the $\fo/p^n$-submodule generated by $p^{m-n}y_m$ and $y_n$ in $A[p^n]$. If $D_{m,n} = A[p^n]$, then the same argument as \cite[Corollary III.8.1.1]{silverman} shows that $\mu_{p^n}\subset F$, which cannot happen for $n\gg 0$. Thus there exists $n_0\in \Z_{\ge 0}$ such that for all $n\ge n_0$, $y_{n_0} = a_np^{n-n_0}y_n$ for some $a_n \in \fo^\times$. Put
		\begin{align*}
			\tilde{y}_n = \begin{cases}
				p^{n_0-n}y_{n_0},& \text{ if }n\le n_0;\\
				a_n y_n, & \text{ if }n>n_0.
			\end{cases}
		\end{align*}
		It follows from our construction that $(\tilde{y}_n)_{n\ge 0} \in \plim_n A[p^n]^{G_F}$; thus we have a nonzero element in $T^{G_F}$, and hence in $V^{G_F}$.
		
		To derive a contradiction, we will now recall some property of the local Galois representation at the bad prime $w$; we thank Ming-Lun Hsieh for explaining this to us. Denote by $I\subset G_F$ the inertia subgroup, and recall $V_f = V(-k/2)$ is Deligne's cohomological Galois representation. Since $w\nmid p$, taking the $I$-invariant commutes with the $p$-cyclotomic twists. Note also that, with $\ell$ being the rational prime under $w$, $I\subset G_{F}\subset G_{\Q_\ell}$ is the full inertia group of $\ell$ as $K$ is unramified at places that divide $N$. Then, as is standard, we have $V^{I} = V_f^I(-k/2)$ is either trivial or one-dimensional. In the latter case, we have $\Frob_\ell$ acts on $V^I$ as $-\ell\varepsilon_{f,\ell}$, following Atkin--Lehner \cite{atkin-lehner} and Carayol \cite{carayol} (see also \cite[Proposition 3.1.(4)]{nekovar92}); here $\varepsilon_{f,\ell}$ is the Atkin--Lehner sign. Therefore, no power of $\Frob_\ell$ is 1 and we have $V^{G_F} = 0$.
\end{proof}

\begin{lem}\label{lem:tamagawa}
    Let $w\mid N$ be a place of $K_n$ that is finitely decomposed in $K_\infty$, and let $w'\mid w$ be a place in $K_\infty$. Denote by $\Gamma_{n,w'}\subset \Gamma_n$ the decomposition group of $w'$. We have
    \begin{align*}
        \#\ker\left(H^1_{/\bff}(K_{n,w},A) \to 
        H^1_{/\bff}(K_{\infty,w'},A)^{\Gamma_{n,w'}}\right)
        \le b_w.
    \end{align*}
    In particular, when $p\nmid b_w$, the above kernel is trivial.
\end{lem}
\begin{proof}
    The proof is a variant of \cite[Lemma 3.3]{greenberg-cetraro}. Recall that by \cite[\S2.1,7, \S2.2.4]{perrin-riou:theorie-iwasawa-hauteurs}, we have $H^1_\bff(K_{\infty,w'},A) = 0$. Next, consider the commutative diagram with exact rows
    \[
        \begin{tikzcd}
            H^1_\bff(K_{n,w},A) \ar[d] \ar[r]
            & H^1(K_{n,w},A) \ar[d] \ar[r]
            & H^1_{/\bff}(K_{n,w},A) \ar[d,"g_{w'}"] \ar[r]
            & 0\\
            0 \ar[r]
            & H^1(K_{\infty,w'},A)^{\Gamma_{n,w'}} \ar[r,equal] 
            & H^1(K_{\infty,w'},A)^{\Gamma_{n,w'}}  \ar[r] & 0
        \end{tikzcd}
    \]
    By {the} snake lemma, we find
    \begin{align*}
        \#\ker(g_{w'})\le \# \ker\left(H^1(K_{n,w},A)\to H^1(K_{\infty,w'},A)^{\Gamma_{n,w'}}\right) =\#H^1(\Gamma_{n,w'},A^{G_{K_{\infty,w'}}}).
    \end{align*}
    As $\Gamma_{n,w'}\ne 0$ is pro-cyclic, if we fix a topological generator $\gamma_{w'}\in \Gamma_{n,w'}$, then
    \begin{align*}
        H^1(\Gamma_{n,w'},A^{G_{K_{\infty,w'}}}) \simeq A^{G_{K_{\infty,w'}}}/(\gamma_{w'}-1)A^{G_{K_{\infty,w'}}}
    \end{align*}
   (see \cite[Exercise 2.2]{greenberg-park-city} for example). Lemma \ref{lem:bad-reduction} tells us that $(A^{G_{K_{\infty,w'}}})_{\div}\subseteq (\gamma_{w'}-1)A^{G_{K_{\infty,w'}}}$. Recall that $ B_w = A^{G_{K_{\infty,w'}}}/\left(A^{G_{K_{\infty,w'}}}\right)_{\div}$. It follows that $A^{G_{K_{\infty,w'}}}/(\gamma_{w'}-1)A^{G_{K_{\infty,w'}}}$ is a quotient of $B_w$. Thus, $\#\ker(g_{w'})\le b_w$.
\end{proof}

\begin{proof}[Proof of Proposition \ref{prop:control}]
    We first note that the injectivity of $\Phi_n$ follows from the inflation-restriction exact sequence
    \begin{align*}
        0\to H^1(\Gamma_n,A^{G_{K_\infty}})\to H^1(K_n,A)\to H^1(K_\infty,A),
    \end{align*}
    and our assumption \eqref{ass:tors}.

    For the surjectivity, we consider the commutative diagram with exact rows
    \[
        \begin{tikzcd}[nodes={inner sep=0.1pt}] 
            \Sel^{\star,\bullet}(K_n,A) \ar[r,hook] \ar[d,"\Phi_n"] 
            & H^1(K_{\Sigma}/K_n,A) \ar[d,"h"] \ar[r] 
            & \displaystyle H^1_{/\star}(K_{n,v},A)\times H^1_{/\bullet}(K_{n,\bar{v}},A) \times \prod_{w\in S_n} H^1_{/\bff}(K_{n,w},A) \ar[d,"g"] \\
            \Sel^{\star,\bullet}(K_\infty,A)^{\Gamma_n} \ar[r,hook] 
            & H^1(K_{\Sigma}/K_\infty,A)^{\Gamma_n} \ar[r] 
            & \displaystyle \left(H^1_{/\star}(K_{\infty,v},A)\times H^1_{/\bullet}(K_{\infty,\bar{v}},A) \times  \prod_{w'\in S_\infty} H^1_{/\bff}(K_{\infty,w'},A)\right)^{\Gamma_n}.
        \end{tikzcd}
    \]
    By the inflation-restriction exact sequence and \eqref{ass:tors} again, we have $\coker(h) = 0$. Thus, by the snake lemma, the surjectivity of $\Phi_n$ would follow from $\ker(g) = 0$. Note that $g$ is given by the local maps
    \begin{align*}
        g_w: H^1_{/\heartsuit}(K_{n,w},A) \to \left(\prod_{w'\mid w} H^1_{/\heartsuit}(K_{\infty,w'},A)\right)^{\Gamma_n},
    \end{align*}
    where
    \begin{align*}
        \heartsuit = 
        \begin{cases}
            \bff & w \mid N;\\
            \star & w = v;\\
            \bullet & w = \barv.\\
        \end{cases}
    \end{align*}
    The inflation-restriction exact sequence and \eqref{ass:tors} imply that $\ker(g_w) = 0$ with $w\mid p$. It remains to show that $\ker(g_w) =0$ for $w\mid N$. For this, we enlarge $g_w$ to
    \begin{align*}
        g_w': H^1_{/\bff}(K_{n,w},A) \to \left(\prod_{w'\mid w} H^1_{/\bff}(K_{\infty,w'},A)\right)^{\Gamma_n} \hookrightarrow
        \prod_{w'\mid w} H^1_{/\bff}(K_{\infty,w'},A)^{\Gamma_{n,w'}}
    \end{align*}
    (notation as in Lemma \ref{lem:tamagawa}). We then have a dichotomy: either $w$ is infinitely decomposed in $K_\infty$, in which case $\ker(g_w') = 0$; or $w$ is finitely decomposed, in which scenario the vanishing of $\ker(g_w')$ remains valid thanks to Lemma \ref{lem:tamagawa} and \eqref{ass:tama}.
\end{proof}

\section{Vanishing of $\Sel^\BDP$ and its consequences}
\label{sec:bdp}

Recall that $z_{f,K}\in H^1_\bff(K,T)$ is the generalized Heegner class over $K$; we use the same letter for its image in $H^1_\bff(K,A)$. The following hypotheses are in effect:
\begin{align}\tag{Loc.prim.}\label{ass:locdiv}
    \loc_v(z_{f,K}) \notin \varpi H^1_\bff(K_v,T);
\end{align}
\begin{align}\tag{Sel.}\label{ass:sel}
    \Sel(K,A) = \fF/\fo\cdot z_{f,K}.
\end{align}
Note also that, by local Tate duality, the assumption \eqref{ass:tors} implies
\begin{align}\tag{Tors'.}\label{ass:tors2}
    H^2(K_{\barv},T) = 0.
\end{align}

The proof of the following proposition is inspired by \cite[proof of Theorem 3.2]{matar-supersingular}.
\begin{prop}\label{prop:bdp}
    Assume \eqref{ass:tors}, \eqref{ass:tama}, \eqref{ass:locdiv} and \eqref{ass:sel}. For all $n\in \Z_{\ge 0}\cup \{\infty\}$,
    \begin{align*}
        \Sel^\BDP(K_n,A) = 0.
    \end{align*}
\end{prop}
\begin{proof}
    By Proposition \ref{prop:control}, we have $\Sel^\BDP(K_n,A) = \Sel^\BDP(K_\infty,A)^{\Gamma_n}$. As $\Sel^\BDP(K_\infty,A)$ is co-finitely generated \cite[Proposition 3]{greenberg:representations}, by Nakayama's lemma it suffices to prove $\Sel^\BDP(K,A) = 0$. Equivalently, we show that
    \begin{align}\label{eq:injectivity}
        \Sel^{\emptyset,\emptyset}(K,A)\to H^1(K_{\barv},A)
    \end{align}
    is injective.

    We start with a bisection of $\Sel^{\emptyset,\emptyset}(K,A)$. We have the following Cassels--Poitou--Tate exact sequence (see \cite[p.~119]{washington:flt} for a detailed discussion over $\Q$; the argument can be applied to general number fields with minor adjustments):
    \begin{align*}
        0\to \Sel(K,A)\to \Sel^{\emptyset,\emptyset}(K,A)\xrightarrow{\psi} H^1_{/\bff}(K_v,A)\times H^1_{/\bff}(K_{\barv},A)\xrightarrow{\theta} \Sel(K,T)^\vee,
    \end{align*}
    where $\Sel(K,T) = \plim_n \Sel(K,A)[p^n]$. This gives rise to the following short exact sequence:
    \begin{align}\label{eq:bisection}
        0\to \Sel(K,A)\to \Sel^{\emptyset,\emptyset}(K,A)\xrightarrow{\psi} \ker(\theta)\to 0.
    \end{align}
    Consequently, the injectivity of \eqref{eq:injectivity} follows from the following two claims:
    \begin{enumerate}
        \item[(a)] if $z\in \Sel^{\emptyset,\emptyset}(K,A)$ is such that $\loc_{\barv}(z) = 0$, then $z\in \Sel(K,A)$;
        \item[(b)] the localization map $\Sel(K,A)\to H^1(K_{\barv},A)$ is injective.
    \end{enumerate}
    We begin with the proof of (b). First, note that the localization map can be decomposed as
    \begin{align}
        \loc_{\barv}: \Sel(K,A)\simeq \fF/\fo\cdot z_{f,K}\xrightarrow{\psi_{0,\barv}} H^1_\bff(K_{\barv},A)\hookrightarrow H^1(K_{\barv},A),
    \end{align}
   where $\psi_{0,\barv}$ is the built-in morphism from the definition of $\Sel(K,A)$, and the first isomorphism is guaranteed by \eqref{ass:sel}. Suppose $z \in \ker(\psi_{0,\barv})$ is of the form $z_{f,K}\otimes a/\varpi^n$ for some $a\in \fo^\times$ and $n\in \Z_{\ge 0}$. By \eqref{ass:tors}, we have an exact sequence
    \begin{align*}
        0\to H^1_\bff(K_\barv,T) \to H^1_\bff(K_\barv,V)\to H^1_\bff(K_\barv,A).
    \end{align*}
    Hence, we may and will regard $H^1_{\bff}(K_{\barv},T)$ as a subgroup of $H^1_{\bff}(K_{\barv},V)$. Since $\psi_{0,\barv}(z) = \loc_{\barv}(z_{f,K})\otimes a/\varpi^n = 0\in H^1_{\bff}(K_{\barv},A)$, we have $\loc_{\barv}(z_{f,K}) \in \varpi^n H^1_{\bff}(K_{\barv},T)$. Now, by applying the complex conjugation to \eqref{ass:locdiv}, we find $\loc_\barv(z_{f,K})\notin \varpi H^1_\bff(K_{\barv},T)$, which forces $n=0$ and hence $z = 0$.

    As for (a), we note that by \eqref{ass:locdiv} and \eqref{ass:sel}, we have $\Sel(K,T) = \plim_n \Sel(K,A)[p^n] = \fo\cdot z_{f,K}$, and that, by definition, $\ker(\theta)\subseteq H^1_{/\bff}(K_v,A)\times H^1_{/\bff}(K_{\barv},A)$ is orthogonal to $\im(\theta^\vee) = \im(\Sel(K,T))\subseteq H^1_\bff(K_v,T)\times H^1_\bff(K_{\barv},T)$. Now, if $z\in \Sel^{\emptyset,\emptyset}(K,A)$ is such that $\loc_{\barv}(z) = 0$, then $\psi(z)\in \ker(\theta)$ is also orthogonal to $H^1_\bff(K_{\barv},T)$, and thus to 
    \begin{align*}
        0\times H^1_\bff(K_{\barv},T) + \im(\Sel(K,T)) \subseteq H^1_\bff(K_v,T)\times H^1_\bff(K_{\barv},T).
    \end{align*}
    By \eqref{ass:tors}, $H^1_\bff(K_v,T)\times H^1_\bff(K_{\barv},T)$ is torsion-free and thus isomorphic to $\fo^2$ (see \cite[Corollary 3.8.4]{bloch-kato:l-fcts-tamagawa-numbers}). Now, \eqref{ass:locdiv} tells us that the projection of $z_{f,K}$ in $H^1_\bff(K_v,T)$ is a generator, whereby $0\times H^1_\bff(K_{\barv},T) + \im(\Sel(K,T)) = H^1_\bff(K_v,T)\times H^1_\bff(K_{\barv},T)$. This shows $\psi(z) = 0$ and hence $z\in \Sel(K,A)$.
\end{proof}
\begin{rem}\label{rem:matar3.2}
    Note that we proved that
    \begin{enumerate}
        \item[(i)] $\Sel(K,A)\to H^1_\bff(K_{\barv},A)$ is an isomorphism;
        
        \item[(ii)] $\ker(\theta)\simeq \fF/\fo$, as its Pontryagin dual is $\big(H^1_\bff(K_v,T)\times H^1_\bff(K_{\barv},T)\big)/\im(\Sel(K,T))$, which is isomorphic to $\fo$ by \eqref{ass:locdiv}. In particular, $\Sel^{\emptyset,\emptyset}(K,A)\simeq (\fF/\fo)^2$ by \eqref{eq:bisection}.
    \end{enumerate}
    Additionally, as one can check, our proof works with $v,\barv$ switched under the same set of assumptions.
\end{rem}

\begin{rem}\label{rem:BDP}
  We discuss a possible alternative proof of Proposition~\ref{prop:bdp} utilizing the Iwasawa main conjecture that relates the BDP $p$-adic $L$-function of \cite{bdp:generalized,miljan,brooks:shimura-curves-l-fcts} (which we denote by $L_p^\BDP\in\widehat{\fo^\mathrm{ur}}[[\Gamma]]$, where $\widehat{\fo^\mathrm{ur}}$ denotes the ring of integers of the completion of the maximal unramified extension of $\fF$) to the BDP Selmer group. Throughout this remark, we assume that the following inclusion holds
\begin{equation}\label{eq:IMC}
\left(L_p^\BDP\right)^2\in \Char_\Lambda\Sel^\BDP(K_\infty,A)^\vee\otimes \widehat{\fo^\mathrm{ur}}.
 \end{equation}
In the ordinary case, this is a consequence of \cite[Theorem~3.5]{longo-vigni-kyoto} under mild hypotheses; see the discussion in \cite[Remark~13]{HL23}, see also \cite[Theorem B]{BCK} and \cite{sweeting} for more recent developments. When $f$ is non-ordinary at $p$, $k=2$ and $N^-=1$, \eqref{eq:IMC} has been studied in \cite{lei-zhao}, which built on the prior work of Kobyashi--Ota \cite{kobayashi-ota}.

One can check, using the interpolation formula for $L_p^\BDP$ at the trivial character (see \cite[Theorem~7.2.4]{JLZ}, \cite[Theorem~5.13]{bdp:generalized} and \cite[Theorem~1.1]{brooks:shimura-curves-l-fcts}),
that $L_p^\BDP$ is a unit under an appropriate hypothesis on the image of the $z_{f,K}$ under the Bloch--Kato logarithm map.
Consequently, \eqref{eq:IMC} implies that $\Sel^\BDP(K_\infty,A)^\vee$ is a pseudo-null $\Lambda$-module. However,  $\Sel^\BDP(K_\infty,A)^\vee$ does not contain any non-trivial pseudo-null submodule (see \cite[Lemma~3.4]{LMX}). Therefore, we deduce $\Sel^\BDP(K_\infty,A)=0$. If \eqref{ass:tama} also holds, we can combine the last equation with the exact control theorem Proposition~\ref{prop:control} to conclude that $\Sel^\BDP(K_n,A)=0$ for all $n$.
\end{rem}

\begin{cor}\label{cor:bdp-sha-trivial}
    Suppose that the assumptions of Proposition \ref{prop:bdp} hold. Then, for all $n\in \Z_{\ge 0}$, the map $\Sel(K_n,A)\to H^1_\bff(K_{n,w},A)$ is an isomorphism for $w\mid p$. Consequently,
     \begin{enumerate}
         \item $\corank_\fo(\Sel(K_n,A)) = p^n$;
         \item the Bloch--Kato--Shafarevich--Tate group $\Sha_{\BK}(K_n,A) = \Sel(K_n,A)/\Sel(K_n,A)_\div = 0$.
     \end{enumerate}
\end{cor}
\begin{proof}
    Without loss of generality, assume $w = \barv$. We show that there is an exact sequence
	\begin{align}\label{eq:CPT-modified}
		\Sel^{\bff,0}(K_n,A)\to \Sel(K_n,A)\xrightarrow{\alpha} H^1_\bff(K_{n,\barv},A)\xrightarrow{\beta} \Sel^{0,\emptyset}(K_n,T)^\vee.
	\end{align}
    
	Consider the following commutative diagram
	\[
		\begin{tikzcd}
			\Sel^{\bff,0}(K_n,A) \ar[r] & \Sel(K_n,A) \ar[r,"\alpha'"] \ar[dr,"\alpha"] & H^1_\bff(K_{n,v},A)\times H^1_\bff(K_{n,\barv},A) \ar[d,"\pi"] \ar[r,"\beta'"] & \Sel^{\emptyset,\emptyset}(K_n,T)^\vee \ar[d,"\iota^\vee"]\\
			&& H^1_\bff(K_{n,\barv},A) \ar[r,"\beta"] & \Sel^{0,\emptyset}(K_n,T)^\vee,
		\end{tikzcd}
	\]
	where $\pi$ is the natural projection and $\iota: \Sel^{0,\emptyset}(K_n,T)\to \Sel^{\emptyset,\emptyset}(K_n,T)$ is the natural injection. 	
	By the Cassels--Poitou--Tate exact sequence,  the first row is exact at $H^1_\bff(K_{n,v},A)\times H^1_\bff(K_{n,\barv},A)$. Suppose $z\in H^1_\bff(K_{n,\barv},A)$ belongs to $\ker(\beta)$. Put $z' = (0,z)\in H^1_\bff(K_{n,v},A)\times H^1_\bff(K_{n,\barv},A)$, then $\iota^\vee\circ\beta'(z') = 0$, so $\beta'(z')$, as a function on $\Sel^{\emptyset,\emptyset}(K_n,T)$, is orthogonal to $\Sel^{0,\emptyset}(K_n,T)$. It follows that $\beta'(z')$ is actually a linear function on $\mathrm{coker}(\iota) \subseteq H^1(K_{n,v},T)$. Since $z'$ has zero $v$-component, we also have $\beta'(z')$ vanishes on $H^1(K_{n,v},T)$. Hence,  $\beta'(z') = 0$. This shows that $z' = \alpha'(w)$ for some $w\in \Sel(K_n,A)$, and thus $z = \pi(z') = \pi\alpha'(w) = \alpha(w)$. This finishes the proof of the exactness of \eqref{eq:CPT-modified}
    
    We now show that the map $\alpha$ is an isomorphism. By \eqref{eq:CPT-modified} and Proposition \ref{prop:bdp}, $\alpha$ is injective. As for the surjectivity, invoking Proposition \ref{prop:bdp} again with $v,\barv$ switched, we have
    \begin{align*}
        \Sel^{0,\emptyset}(K_n,A) = \bigcup_{m> 0} \Sel^{0,\emptyset}(K_n,A)[p^m] = 0.
    \end{align*}
    Thus, $\Sel^{0,\emptyset}(K_n,A)[p^m] = 0$ for all $m$, whereby $\Sel^{0,\emptyset}(K_n,T) = \plim_m \Sel^{0,\emptyset}(K_n,A)[p^m] = 0$.

    Finally, we turn our attention to the ``therefore'' part. For (1), we have $\corank_\fo(H^1_\bff(K_{n,w},A)) = p^n$ by \cite[Corollary~3.8.4]{bloch-kato:l-fcts-tamagawa-numbers} and the base change property $(B_{\rm dR}\otimes V)^{G_{K_{n,w}}} = (B_{\rm dR}\otimes V)^{G_{K_w}}\otimes_{K_w}K_{n,w}$. The assertion (2) follows from the fact that $H^1(K_{n,w},A)$ is divisible (see \cite[proof of Theorem 2.9]{greenberg-park-city}).
\end{proof}

\section{Results on universal norms}
\label{sec:universal-norms}
Throughout \S\ref{sec:universal-norms}, we let $w\in\{v,\barv\}$ be a place of $K$. The objective of this section is to study the structure of $\Sel(K_\infty,A)$. It follows from Corollary~\ref{cor:bdp-sha-trivial} that under the assumptions of Proposition~\ref{prop:bdp},
\begin{equation}
\Sel(K_\infty,A)=\varinjlim_n \Sel(K_n,A)=\varinjlim_n H^1_\bff(K_{n,w},A).
    \label{eq:UN}
\end{equation}
We shall prove part (3) of Theorem~\ref{mainthm}, which describes the structure of this $\Lambda$-module. We do so via the theory of local universal norms à la Perrin-Riou \cite{perrinriou00} who studied modules over cyclotomic towers analogous to the second direct limit in \eqref{eq:UN}.

\begin{lem}
    There is a $\Lambda$-isomorphism
\[
\left(\varinjlim_n H^1_\bff(K_{n,w},A)\right)^\vee\simeq \varprojlim_n H^1(K_{n,w},T)/\varprojlim_n H^1_\bff(K_{n,w},T).
\]
\end{lem}
\begin{proof}
    By local Tate duality, there is a perfect pairing
    \[
   \langle-,-\rangle:\varinjlim_n H^1(K_{n,w},A)\times\varprojlim_n H^1(K_{n,w},T)\rightarrow\fF/\fo.
    \]
    As explained in \cite[\S2.1.8]{perrin-riou:theorie-iwasawa-hauteurs}, $\langle-,-\rangle$ induces a perfect pairing $$\left(\varinjlim_n H^1(K_{n,w},A)/\varinjlim_n H_\bff^1(K_{n,w},A)\right)\times\displaystyle\varprojlim_n H_\bff^1(K_{n,w},T)\rightarrow \fF/\fo.$$ In other words, $\displaystyle\varinjlim_n H_\bff^1(K_{n,w},A)$ and $\displaystyle\varprojlim_n H_\bff^1(K_{n,w},T)$ are the annihilator of each other under $\langle-,-\rangle$, from which the lemma follows.
\end{proof}

We introduce the notation that will be used in our proof of Theorem~\ref{mainthm}(3). Let $G$ be a topological group that is isomorphic to $\Zp$. Let $\gamma$ be a topological generator of $G$. We write $\sH_\infty(G)$ for the algebra of tempered distributions on $G$. More explicitly, it is the set of power series $f(\gamma-1)$ where $f(X)\in\Qp[[X]]$ that converges on the open unit disk. Given a real number $r>0$, we write $\sH_r(G)$ for the subset of $\sH_\infty(G)$ consisting of the elements that are $O(\log_p^r)$, i.e., those arise from power series $f(X)=\sum_{n\ge0} c_nX^n$ satisfying $\sup\{|c_n|_p/n^r\}<\infty$.
 Let $\tilde\fo$ denote the ring of integers of the completion of the maximal unramified extension of $\fF$. We write $\tilde\sH_\infty(G)=\sH_\infty(G)\otimes_{\Zp}\tilde\fo$ and  $\tilde\sH_r(G)$ is defined similarly.

Given a $G_{K_w}$-representation $M$ and a $p$-adic Lie extension $\cK$ of $K_w$, we write 
\[
\HIw(\cK,M)=\varprojlim H^1( K',M),
\]
where $M$ runs through finite extension of $K_w$ contained inside $\cK$ and the connecting maps are corestrictions.

Let $K_{\cyc,w}$ be the cyclotomic $\Zp$-extension of $K_w$ and let $\cK$ denote the compositum of $K_{\cyc.w}$ and $K_{\infty,w}$.
We write $\Gamma_\cyc=\Gal(K_{\cyc,w}/K_w)$ and $\Gamma'=\Gal(\cK/K_{\cyc,w})$, the latter of which can be identified with the Galois group of the unramified $\Zp$-extension of $K_w$.
Let $\alpha$ and $\beta$ be the two roots of $X^2-a_p(f)X+p^{k-1}$. Let $r=\max(\ord_p(\alpha),\ord_p(\beta)$. Since the Hodge--Tate weights of $T(k/2-1)$ are $0$ and $k-1$, there is a two-variable Perrin-Riou map
$$\cL_{\cK,T(k/2-1)}:\HIw(\cK,T(k/2-1))\rightarrow \tilde\fo[[\Gamma']]\widehat{\otimes}\sH_r(\Gamma_\cyc)\otimes\Dcris(V(k/2-1))$$
defined in \cite[Theorem 4.7]{LZ14} (the fact that the image consists of elements that are $O(\log_p^r)$ is proven in Proposition 4.8 of \textit{op.~cit.}).
There is a natural isomorphism 
$\HIw(\cK,T)\otimes \Zp(k/2-1)\stackrel{\sim}{\to}\HIw(\cK,T(k/2-1))$ (see \cite[Corollary A.4.4]{perrinriou95}). This induces a corresponding Perrin-Riou map
$$\cL_{\cK,T}:\HIw(\cK,T)\rightarrow \tilde\fo[[\Gamma']]\widehat{\otimes}\sH_r(\Gamma_\cyc)\otimes\Dcris(V).$$

As in \cite[Theorem A.2]{lei-zhao}, we may quotient out by the kernel of the projection map
\[
\tilde\fo[[\Gal(\cK/K_w)]]\rightarrow \tilde\fo[[\Gamma]],
\]
resulting in
\[
\cL_{K_{\infty,w},T}:\HIw (K_{\infty,w},T)\rightarrow \tilde\sH_r(\Gamma)\otimes\Dcris(V).
\]

\begin{thm}\label{thm:lambda-ranks}
Assume that \eqref{ass:tors} holds. Then:
\item[i)] the inverse limit $\HIw(K_{\infty,w},T)$ is free of rank two over $\Lambda$;
\item[ii)]    if $f$ is non-ordinary at $p$, then $\displaystyle\varprojlim_n H^1_\bff(K_{n,w},T)=0$;
\item[iii)] if $f$ ordinary at $p$, then $\HIw(K_{\infty,w},T)/\varprojlim_n H^1_\bff(K_{n,w},T)$ is free of rank one over $\Lambda$. 
\end{thm}
\begin{proof}
We first prove part i). By \cite[Proposition 2.1.3 and 2.1.6]{perrin-riou:theorie-iwasawa-hauteurs} and Lemma~\ref{lem:torsion}, we see that $\HIw(K_{\infty,w},T)$ is of $\Lambda$-rank 2 and isomorphic to $\Hom_{\Lambda}(\cZ,\Lambda)$ for some finitely generated $\Lambda$-module $\cZ$. We contend that $\Hom_{\Lambda}(\cZ,\Lambda)$ is a free $\Lambda$-module, which is enough to establish i). Indeed, let $\cN$ be the maximal torsion $\Lambda$-submodule of $\cZ$, then $\Hom_{\Lambda}(\cZ/\cN,\Lambda) = \Hom_\Lambda(\cZ,\Lambda)$, so we may assume $\cZ$ is torsion-free over $\Lambda$. As $\cZ$ is finitely generated, there exists a pseudo-null $\Lambda$-module $\cM$ such that
\begin{align*}
    0\to \cZ \to \Lambda^{\oplus 2} \to \cM\to 0
\end{align*}
is exact. Taking $\Hom_\Lambda(-, \Lambda)$, we have the exact sequence $0\to \Hom_\Lambda(\Lambda^{\oplus 2},\Lambda)\to \Hom(\cZ,\Lambda)\to \mathrm{Ext}^1_{\Lambda}(\cM,\Lambda)$. Thus, we are left to show $\mathrm{Ext}^1_{\Lambda}(\cM,\Lambda) = 0$; namely, any $\Lambda$-module extension $0\to \Lambda\to J\to \cM\to 0$ splits. 

Since $\cM$ is pseudo-null, there exists a non-zero element $g\in \Lambda$ that is an annihilator of $\cM$. After multiplying by $g$ on the short exact sequence of $0\to \Lambda\to J\to \cM\to 0$ and applying snake lemma, we have another exact sequence $0\to J^{g=0}\to \cM\to \Lambda/(g)$. Since $\Lambda/(g)$ has no pseudo-null submodule, we conclude that there is a canonical lifting $\cM\simeq J^{g=0}\hookrightarrow J$; hence, the extension is trivial, which concludes the proof of part i).

We turn our attention to ii). In particular, we assume that $f$ is non-ordinary at $p$.
If $\bz=(z_n)\in\varprojlim H^1_\bff(K_{n,w},T)$, then the image of $z_n$ under the Bloch--Kato dual exponential map is zero. Thus, the interpolation formula of the Perrin-Riou map (see \cite[Theorem~4.15]{LZ14}) implies that $\cL_{K_{\infty,w},T}(\bz)$ is an element of the following set:
\begin{align*}
\cA=\big\{g\in \tilde\sH_r\otimes\Dcris(V):g(\chi^{-j}\theta)\in K_n\otimes \varphi^n\Fil^j\Dcris(V),g(\theta)=0,\\
\text{for all finite characters $\theta$ of $\Gamma$ and }j\le k/2\big\}.
\end{align*}
It follows from the calculations in \cite[\S3]{perrinriou00} that when $V|_{G_{K_w}}$ is irreducible, the set $\cA=0$, which in turn implies that $\bz=0$ by Proposition 2.5.4 of \textit{op.~cit.}

We now prove iii). Since $f$ is assumed to be ordinary at $p$, $T|_{G_{K_w}}$ admits a one-dimensional sub-representation $T^+$ so that 
\[
H^1_\bff(L',T)=H^1(L',T^+)
\]
for any finite extension $L'$ of $K_w$ (see \cite[Remark~2.3.4]{perrin-riou:theorie-iwasawa-hauteurs}, which says that $
H^1_\bff(L',T)$ and $H^1(L',T^+)$ agree up to torsion; Lemma~\ref{lem:torsion} tells us that the torsion subgroups are trivial). Next, we claim that the natural map $\HIw(K_{\infty,w},T)/\HIw(K_{\infty,w},T^+)\to \HIw(K_{\infty,w},T^-)$ is an isomorphism. By dualizing, this boils down to the exactness of the following sequence
\begin{align*}
    0\to H^1(K_{\infty,w},A^+)\to H^1(K_{\infty,w},A)\to H^1(K_{\infty,w},A^-)\to 0,
\end{align*}
which in turn is guaranteed by Lemma \ref{lem:torsion}(b) (note that by local duality, $H^2(K_{\infty,w},A^+) = H^0(K_{\infty,w},T^-)^\vee$). Now, we may proceed as in the proof of i) to conclude that $\HIw(K_{\infty,w},T^-)$ is free of rank one over $\Lambda$, again using Lemma \ref{lem:torsion}(b).
\end{proof}

\begin{rem}
The calculations in this section are carried out purely locally at $p$ and do not depend on the hypotheses \eqref{ass:tama} and \eqref{ass:locdiv} assumed in Proposition~
\ref{prop:bdp}. Their sole role is in enabling the passage between the Selmer group and the local cohomology groups via \eqref{eq:UN}.
\end{rem}

\appendix

\section{Growth of Selmer groups: the case of $a_p=0$}
\label{app}
\subsection{Summary of results}
\label{sec:appendix-intro}
In this appendix, following the line of argument of Matar presented in \cite{matar-supersingular}, we give an alternative proof of Theorem~\ref{mainthm} for modular forms  $f$ with $a_p(f)= 0$ utilizing the plus and minus Selmer groups. As in \S\ref{sec:heegner}, when $N^- > 1$, we will need to assume that the weight $k$ is greater than 2, as this is the setting of \cite{magrone:generalized-heegner-cycles-quaternionic}.

We will keep the same notation from the main body of the article, with the following additions
\begin{itemize}
	\item If $R$ is an $\fo$-algebra, we write $\Lambda_R = \Lambda\otimes_{\fo}R$.
	
	\item From the generalized Heegner classes $z_{f,p^n}$ \cite{castella-hsieh:heegner-cycles-p-adic-l-functions,magrone:generalized-heegner-cycles-quaternionic}, we form
	\begin{align*}
		\alpha_n = \cores_{K[p^{n+1}]/K_n}(z_{f,p^{n+1}})\in H^1(K_n,T).
	\end{align*}

    \item For $n\ge 1$, we denote $\Phi_n(X) = \frac{X^{p^n}-1}{X^{p^{n-1}}-1}$, $\omega_n(X) = (X+1)^{p^n}-1$,
    \begin{align*}
	   \tilde{\omega}_n^+(X) = \prod_{2\le i\le n,i\equiv 0\bmod 2} \Phi_i(X+1),\quad \tilde{\omega}_n^-(X) = \prod_{1\le i\le n,i\equiv 1\bmod 2} \Phi_i(X+1),
    \end{align*}
    and $\omega_n^\pm(X) = X\tilde{\omega}_n^\pm(X)$. 
\end{itemize}

Our main result is the following
\begin{mainThm}\label{thm:B}
Let $f$ be a normalized cuspidal eigen-newform of even weight $k$ and level $\Gamma_0(N)$. Let $K$ be an imaginary quadratic field with coprime-to-$Np$ discriminant $d_K\ne -3,-4$ and satisfying \eqref{ass:genHeeg}. Assume moreover $N\ge 5$ if $N^-=1$, and $N^+> 3$ and $k\ge4$ if $N^-\ne 1$. Let $p$ be a prime such that:
	\begin{enumerate}
		\item[(i)] $p$ splits in $K$, and does not divide the class number $h_K$;
		\item[(ii')] $p\nmid 2(k-1)!N\varphi(N)$;
		\item[(iii')] $a_p(f) = 0$;
		\item[(iv)] {For any $w\mid N$ in $K$ and any $w'\mid w$ in $K_\infty$, the group $A^{G_{K_{\infty,w'}}}$ is divisible;}
		\item[(v')] The class $z_{f, K}$ is globally primitive, i.e., $z_{f,K}\notin \varpi H^1(K,T)$. Moreover, $p\alpha_1 \ne p^{(k-2)/2}(p-1)/2\cdot \alpha_0$ and $p\alpha_2 \ne -p^{k-2}\alpha_0$;
		\item[(vi)] $\Sel(K,A)\simeq \fF/\fo\cdot z_{f,K}$;
        \item[(vii)] $p$ is unramified in $\fF$.
	\end{enumerate}
	Then:
	\begin{enumerate}
		\item[(I)] the Selmer modules $\Sel^+(K_\infty,A)$ and $\Sel^-(K_\infty,A)$, whose definitions are recalled in \S\ref{subsec:selmer-control}, are cofree of corank one over $\Lambda$;
		\item[(II)] for all $n\ge 0$, $\corank_{\fo}(\Sel(K_n/A)) = p^n$;
		\item[(III)] the Shafarevich--Tate groups $\Sha_{\BK}(K_n,A)$ and $\Sha_{\AJ}(K_n,A)$ introduced in \S\ref{sec:heegner} coincide. In particular, $\Sha_{\AJ}(K_n,A)$ is finite for all $n\ge 0$.
	\end{enumerate}
	Suppose in addition that $\loc_v(z_{f,K})\notin \varpi H^1_\bff(K_v,T)$, then we have
	\begin{enumerate}
		\item[(IV)] $\Sel(K_\infty,A)^\vee$ is free of rank 2 over $\Lambda$;
		\item[(V)] $\Sha_{\AJ}(K_n,A) = 0$ for all $n\in\Z_{\ge 0}\cup \{\infty\}$.
	\end{enumerate}
\end{mainThm}
\begin{rem}
		Some comments for the assumptions:
		\begin{enumerate}
            \item we need the Fontaine--Laffaille condition $p> k-1$ in (ii') so as to construct a strongly divisible lattice {(in the sense of \cite[D\'efinition 7.7]{FL}; see also  \cite[Theorem 4.3]{bloch-kato:l-fcts-tamagawa-numbers})} inside the crystalline Dieudonn\'e module of $V$; this allows us to define integral Coleman maps in \S\ref{sec:appendix-review-plus-minus};
            
			\item in the supersingular setting, the assumption (iii) in Theorem \ref{mainthm}, as well as its consequence Lemma \ref{lem:torsion}.(a), is automatically valid (see Lemma \ref{lem:lem4.4});

            \item {we continue to assume the $p$-freeness of Tamagawa numbers, condition (iv), to ensure exact control theorems of plus/minus Selmer groups as well as that of their intersection (see Proposition \ref{prop:pm-control}, Lemma \ref{lem:sel-one-intersection} below);}
            
			\item condition (v') allows us to study Selmer groups explicitly by exploiting the Heegner modules sitting inside, an approach also employed by \cite{perrin-riou87-bullSMF,longo-vigni-plus-minus,matar-supersingular}. Note that if $k=2$, then the conditions on $\alpha_1,\alpha_2$ are superfluous as $\alpha_0 = -2z_{f,K}\notin pH^1(K,T)$;
			
			\item under the assumption $z_{f,K}\notin \varpi H^1(K,T)$ from (v'), we see that (vi) actually asserts that $z_{f,K}$ is a generator of $\Sel(K,A)$, namely, the map $\fF/\fo\to \Sel(K,A)$ sending $t$ to $z_{f,K}\otimes t$ is an isomorphism;

            \item the additional assumption (vii) is required for our argument in \S\ref{sec:appendix-nontorsion}, which guarantees the irreducibility of cyclotomic polynomials $\Phi_n(X)$.
		\end{enumerate}
	\end{rem}

\subsubsection{Organization}

This appendix is arranged as follows: In \S\ref{sec:appendix-review-plus-minus} we recall the definition of plus/minus Selmer groups and establish some basic properties of them in parallel with those studied by Kobayashi \cite{kobayashi03}. In \S\ref{sec:appendix-nontorsion}, we adapt an argument of Matar to prove the nontorsionness (Corollary \ref{cor:nontorsion}) and the lower bound of the corank of $\Sel(K_n,A)$ (Proposition \ref{prop:lower-bound}). In \S\ref{sec:appendix-proof-main}, we replicate Matar's argument in our setting and establish the $\Lambda$-structures of $\Sel^\pm(K_\infty,A)^\vee$ and the upper bound of $\corank_\fo(\Sel(K_n,A))$. Combined with the study of generalized Heegner modules, these are enough to deduce the finiteness of Shafarevich--Tate groups. Finally, assuming $\loc_v(z_{f,K})\notin \varpi H^1_\bff(K_v,T)$, we borrow some results from \S\ref{sec:bdp} to deduce the rest of the theorem.

\subsection{Review of plus/minus theory}
\label{sec:appendix-review-plus-minus}
\subsubsection{Setup}
Throughout the appendix, we assume $p\nmid 2N\varphi(N)(k-1)!h_K$, $p$ splits in $K$ and $a_p(f)=0$. We recall here Lei's theory of plus/minus Coleman maps for modular forms \cite{lei10, lei-compositio}. As the Galois representation $V$ is crystalline, we can form the corresponding Dieudonn\'e module $D(V) =\Dcris(V)$, which is equipped with the decreasing de Rham filtration and the semi-linear operator $\varphi$. For $i\in \Z$, we write $D^i(V)$ for the $i$-th filtration of $D(V)$. In addition, since {we assume the Fontaine--Laffaille condition $p>k-1$ holds, there is a strongly divisible lattice $D(T)\subset D(V)$ (see \cite[D\'efinition 7.7]{FL} or \cite[Theorem 4.3]{bloch-kato:l-fcts-tamagawa-numbers}), which inherits a filtration from $D(V)$. We denote this filtration by $D^\bullet(T)$}.

Throughout \S\ref{sec:appendix-review-plus-minus} we will fix $w\in \{v,\bar{v}\}$. As $K_{\infty,w}/K_w$ is a totally ramified $\Z_p$-extension of $K_w = \Q_p$, we have the Perrin-Riou regulator as constructed in \cite[\S2]{lei10}:
\begin{align*}
	\cL_T: H^1(K_{\infty,w},T) = \plim_n H^1(K_{n,w},T) \longrightarrow D(V)\otimes \sH_\infty,
\end{align*}
 where $\sH_\infty$ is the algebra of tempered distributions on $\Gamma\simeq\gal(K_{\infty,w}/K_w)$, which we may identify with the set of power series in $\fF[[X]]$ that converge on the $p$-adic open unit disc $\{z\in \C_p: |z|<1\}$.

\subsubsection{Coleman maps}
As $V$ has Hodge--Tate weights $\{r,1-r\}$, where $r=k/2$, we see that
\begin{align*}
	D^i(T) = \begin{cases}
		D(T), & i\le -r;\\
		\fo \cdot \omega, & -r<i\le r-1;\\
		0, & i\ge r.
	\end{cases}
\end{align*}
Here, $\omega$ is a chosen generator of $D^0(T)$. 

\begin{lem}
	We have $\varphi(\omega)/p^{r-1}\in D(T)$, and $D(T)$ is freely spanned by $\omega$ and $\varphi(\omega)/p^{r-1}$. 
\end{lem}
\begin{proof}
	We argue in the same way as \cite[Lemma 3.1]{LLZ17}. The Fontaine--Laffaille condition guaranties the following:
	\begin{itemize}
		\item $\varphi(\omega) \in p^{r-1}D(T)$ since $\omega\in D^{r-1}(T)$.
		\item There exists $\xi \in D(T)$ such that $D(T) = \fo\xi \oplus \fo\omega$. 
		\item $D(T) = \fo\cdot p^{1-r}\varphi(\omega) + p^r\varphi(D(T))$.
	\end{itemize}
	Let $D$ be the $\fo$-submodule of $D(T)$ spanned by $\omega$ and $\varphi(\omega)/p^{r-1}$. The second and third properties then implies $\xi \in D + \fo\cdot p^r\varphi(\xi)$. Iterating, we find $\xi\in D + p^r\varphi(D) + \fo\cdot p^k\varphi^2(\xi)$. As $V = V_f(r)$, we have $\varphi^2 = -p^{-1}$. So $p^r\varphi(D) \subset \fo \cdot p^r\varphi(\omega) + \fo\omega\subset D$, and thus $\xi\in D+\fo p^{k-1}\xi$. This forces $\xi\in D$. 
\end{proof}

Now, put $v_1 = \omega, v_2 = \varphi(\omega)/p^{r-1}$. Then under this basis, $\cL_T$ admits the following decomposition for all $z\in H^1(K_{\infty,v},T)$ \cite[Definition 2.8]{lei10}:
\begin{align*}
	\cL_T(z) = \begin{bmatrix}
	    v_1&v_2
	\end{bmatrix} \mat{\log_{p,k}^+}{0}{0}{\log_{p,k}^-} \col{\Col^+(z)}{\Col^-(z)}.
\end{align*}
Here,
\begin{itemize}
	\item for $\varepsilon\in \{\pm 1\}$,
	\begin{align*}
		\log_{p,k}^\varepsilon = \prod_{1-r\le j\le r-1}\frac{1}{p}\prod_{\substack{n\ge 1\\ (-1)^n = \varepsilon}} \frac{\Phi_n(\gamma^{-j}(1+X))}{p}.
	\end{align*}
	As shown by Pollack \cite[Lemma 4.1]{pollack:p-adic-l-fct-mod-form-supersugular-primes}, both products converge in $\Q_p[[X]]$;
	
	\item the Coleman maps, $\Col^\pm$, are $\Lambda$-morphisms from $H^1(K_{\infty,w},T^*)$ to $\Lambda_\fF$. By \cite[Proposition 2.20]{kazim-antonio17}, it is known that the image of $\Col^\pm$ is a finite-index submodule of some free $\Lambda$-submodule of rank 1 inside $\Lambda_\fF$.
\end{itemize}

Next, we recall the following lemma that will be crucially used in our proof of the main theorem.
\begin{lem}\label{lem:lem4.4}
	We have $A[\varpi]^{G_{K_{\infty,w}}} = 0$, and hence $A^{G_{K_{\infty,w}}} = 0$.
\end{lem}
\begin{proof}
This follows from the same proof of \cite[Lemma 4.4]{lei-zhao} and \cite[Lemma 4.4]{lei-compositio}. By \cite[Theorem 2.6]{edixhoven92}, the inertia group at $p$ acts on $A[\varpi]$ as $\begin{bmatrix}
	    \psi^{k-1} &0\\ 0&\psi'^{k-1}
	\end{bmatrix}$, where $\psi$ and $\psi'$ are fundamental characters of level $2$. As $\psi^{k-1}$ and $\psi'^{k-1}$ are non-trivial characters, it follows that $A[\varpi]^{G_{K_{\infty,w}}} = 0$. 
\end{proof}

Thanks to Lemma \ref{lem:lem4.4}, we have an identification
\begin{align*}
	H^1(K_{n,w},A) = H^1(K_{\infty,w},A)^{\Gamma_n}
\end{align*}
for all $n\in \Z_{\ge 0}$. By local Tate duality, we derive
\begin{align*}
	H^1(K_{n,w},T) = H^1(K_{\infty,w},T)_{\Gamma_n}.
\end{align*}
As such, we can define $\fCol_n^\pm$ to be $\Col^\pm$ modulo $\Gamma_n$, namely the induced map
\begin{align*}
	\fCol_n^\pm: H^1(K_{n,w},T)\to \Lambda_{\fF}/\omega_n\Lambda_{\fF}.
\end{align*}

\subsubsection{Plus/minus local subgroups}

We put
\begin{align*}
	H^1_{\pm}(K_{\infty,w},T) = \ker(\Col^\pm)\subset H^1(K_{\infty,w},T),
\end{align*}
and write $H^1_{\pm}(K_{\infty,w},A)\subset H^1(K_{\infty,w},A)$ for its orthogonal complement under the local Tate pairing; put also
\begin{align*}
	H^1_\pm(K_{n,w},A) = H^1(K_{\infty,w},A)^{\Gamma_n}.
\end{align*}
By local duality, we have
\begin{align*}
	H^1_\pm (K_{n,w},A) = \ker(\fCol_n^\pm)^\perp.
\end{align*}
We recall the following description given in \cite[Proposition 3.3]{lei10}; note that for $m\in \Z_{\ge 0}$, the restriction map $H^1(K_{m,w},A)\to H^1(K_{m+1,w},A)$ is injective, by Lemma \ref{lem:lem4.4}.
\begin{lem}\label{lem:even-odd-local-condition}
	Let $\varepsilon\in \{\pm 1\}$. The local subgroup $H^1_{\varepsilon}(K_{n,w},A)$ coincides with
	\begin{align*}
		\left\{z\in H^1_\bff(K_{n,w},A): \cores_{K_n/K_{m+1}}(z)\in H^1(K_{m,v},A)\text{ for all }m\in \Z_{\ge 0}, m<n, (-1)^m = \varepsilon\right\}.
	\end{align*}
\end{lem}

\begin{cor}\label{cor:fCol-Col}
	The image of $\fCol_n^\pm$ is annihilated by $\omega_n^\pm$.
\end{cor}
\begin{proof}
	Let $\varepsilon\in \{\pm 1\}$. Since $\fCol_n^\varepsilon$ is valued in $\Lambda_{\fF}/\omega_n\Lambda_{\fF}$, it suffices to show that there exists $t\in \Z_{\ge 0}$ such that $p^t \omega_n^\varepsilon$ annihilates $\im(\fCol_n^\varepsilon)$. Or, by duality, any $z\in H^1_{\varepsilon}(K_{m,v},A)$ is killed by $p^t \omega_n^\varepsilon$. By Lemma \ref{lem:even-odd-local-condition}, we have
	\begin{align*}
		p^{t(n,\varepsilon)}\omega_n^\varepsilon\cdot z = X\cdot \tr_{K_n/K}(z) = 0,
	\end{align*}
	where
    \begin{align*}
        t(n,\varepsilon) = \begin{cases}
            n/2, & \text{if }2\mid n;\\
            (n+1)/2, & \text{if }2\nmid n\text{ and }\varepsilon = +1;\\
            (n-1)/2, & \text{otherwise.}
        \end{cases}
    \end{align*}
\end{proof}

We will thus denote by $\Col_n^\pm$ the composition
\begin{align*}
	\Col_n^\pm: H^1(K_{n,w},T) \xrightarrow{\fCol_n^\pm}\Lambda_{\fF}/\omega_n\Lambda_{\fF} \surj \Lambda_{\fF}/\omega_n^\pm\Lambda_{\fF},
\end{align*}
where we identify $\Lambda_\fF/\omega_n \simeq \Lambda_{\fF}/\omega_n^\pm \oplus \Lambda_\fF/\tilde{\omega}_n^\mp$ by the Chinese remainder theorem and the second map is the natural projection.
\begin{prop}
	We have $\ker(\Col_n^\pm) = \ker(\fCol_n^\pm)$.
\end{prop}
\begin{proof}
	We will prove the plus part; the minus part can be handled similarly. By definition, it boils down to establishing the reverse inclusion $\ker(\Col_n^+)\subseteq \ker(\fCol_n^+)$. Applying the Chinese remainder theorem, we have $F,G\in \Z_p[[X]]$ and $m\in \Z_{\ge 0}$ such that $F\omega_n^+ + G\tilde{\omega}_n^- = p^m$; note that this forces $G$ to be prime to $\omega_n^+$ in $\Lambda_{\fF}$. Now, suppose $z\in \ker(\Col_n^+)$. It follows from Corollary \ref{cor:fCol-Col} that
	\begin{align*}
		p^m \cdot \fCol_n^+(z) = G\tilde{\omega}_n^-\cdot\fCol_n^+(z).
	\end{align*}
	So, $G(X)\tilde{\omega}_n^-(X)\fCol_n^+(z) = 0\in \Lambda_{\fF}/\omega_n^+\Lambda_{\fF}$. As $\Lambda_{\fF}$ is a unique factorization domain, this shows $\fCol_n^+(z)\in \omega_n^+\Lambda_{\fF}$, whereby $p^m\fCol_n^+(z) = G\tilde{\omega}_n^-\cdot\fCol_n^+(z) = 0\in \Lambda_{\fF}/\omega_n\Lambda_{\fF}$. That is, $z\in \ker(\fCol_n^+)$.
\end{proof}

\begin{rem}
	If we consider $\Col^\pm$ to be valued in $\fF[[\Gamma]]$, then from our construction we have $\Col^\pm = \plim_n \Col_n^+$. In particular, this equality holds in $\Lambda_{\fF}$; \textit{cf.}, \cite[Definition 8.22]{kobayashi03}.
\end{rem}

\subsubsection{Selmer groups and their control theorems}
\label{subsec:selmer-control}

For $n\in \Z_{\ge 0}\cup\{\infty\}$, we define
\begin{align*}
	\Sel^\pm(K_n,A) = \ker\Bigg(H^1(K_n,A)\to \prod_{w\mid p} \frac{H^1(K_{n,w},A)}{H^1_\pm(K_{n,w},A)}\times \prod_{w\nmid p} \frac{H^1(K_{n,w},A)}{H^1_\bff(K_{n,w},A)}\Bigg).
\end{align*}
As both $\Sel^\pm(K_\infty,A)$ are contained in $\Sel(K_\infty,A)$, they are cofinitely generated $\Lambda$-modules \cite[Proposition 3]{greenberg:representations}.

\begin{rem}
	From Lemma \ref{lem:even-odd-local-condition}, we see that
	\begin{align*}
		\Sel^{\pm}(K,A) = \Sel(K,A).
	\end{align*}
\end{rem}

Next, we consider a special case of the control theorem of Ponsinet \cite[Lemma 2.3]{ponsinet20}:
\begin{prop}\label{prop:pm-control}
	Assume \eqref{ass:tama} holds. For $n\ge 0$, the map
	\begin{align*}
		\Sel^\pm(K_n,A) \to \Sel^\pm(K_\infty,A)^{\omega_n = 0}
	\end{align*}
	is an isomorphism.
\end{prop}
\begin{proof}
	By Lemma \ref{lem:lem4.4} the map is injective. For the surjectivity, the proof is more or less in the same lines as those of Proposition \ref{prop:control}, except that at $w\mid p$, one observes instead
    \begin{align*}
        \frac{H^1(K_{n,w},A)}{H^1_\pm(K_{n,w},A)}\to \left(\frac{H^1(K_{\infty,w},A)}{H^1_\pm(K_{\infty,w},A)}\right)^{\Gamma_n}
    \end{align*}
    is injective (\textit{cf., loc.~cit.}).
\end{proof}

Next, after Iovita--Pollack \cite[Definition 7.3]{iovita-pollack}, for $n\in \Z_{\ge 0}\cup\{\infty\}$ we define
\begin{align*}
	\Sel^1(K_n,A) = \ker\Big(\Sel(K_n,A) \to \prod_{w\mid p} \frac{H^1(K_{n,w},A)}{H^1_\bff(K_w,A)}\Big).
\end{align*}
We first note the following
\begin{lem}\label{lem:sel-one-intersection}
	The intersection $\Sel^+(K_n,A)\cap \Sel^-(K_n,A)$ equals $\Sel^1(K_n,A)$.
\end{lem}
\begin{proof}
	We have $\Sel^1(K_n,A)\subseteq \Sel^+(K_n,A)\cap \Sel^-(K_n,A)$ by definition. Now suppose $x\in \Sel^+(K_n,A)\cap \Sel^-(K_n,A)$. Then for $w\mid p$, we have $\loc_w(x) \in H^1_+(K_{n,w},A)\cap H^1_-(K_{n,w},A)$. We claim that $\loc_w(x)\in H^1_\bff(K_w,A)$. Otherwise, let $t\in \Z_{\ge 1}$ be smallest such that $\loc_v(x)\in H^1_\bff(K_{t,w},A)\setminus H^1_\bff(K_{t-1,w},A)$. As $\loc_w(x)\in H^1_{(-1)^{t-1}}(K_{n,w},A)$, by Lemma \ref{lem:even-odd-local-condition}, we have
    \begin{align*}
        \cores_{K_n/K_t}(\loc_w(x))\in H^1_\bff(K_{t-1,w},A).
    \end{align*}
    Since 
	\begin{align*}
		\loc_v(x) = \frac{1}{p^{n-t}}\cores_{K_n/K_t}(\loc_w(x)),
	\end{align*}
	we find $p^{n-t}\loc_w(x)\in H^1_\bff(K_{t-1,w},A)$. However, as $H^1_\bff(K_{t-1},A)$ is divisible (see \cite[proof of Theorem 2.9]{greenberg-park-city}), we deduce that $\loc_w(x)\in H^1_\bff(K_{t-1},A)$, which is a contradiction.
\end{proof}

\begin{lem}\label{lem:control-sel-one}
	Assume \eqref{ass:tama} holds. For $n\in \Z_{\ge 0}$, the injection
	\begin{align*}
		\Sel^1(K_n,A) \hookrightarrow \Sel^1(K_\infty,A)^{\Gamma_n}
	\end{align*}
	is an isomorphism.
\end{lem}
\begin{proof}
	The proof is similar to that of Proposition \ref{prop:control}. In the snake diagram, at a place $w\mid p$ we use the identification 
		\begin{align*}
			\frac{H^1(K_{n,w},A)}{H^1_\bff(K_w,A)} = \frac{H^1(K_{\infty,w},A)^{\Gamma_n}}{H^1_\bff(K_w,A)},
		\end{align*}
		valid by Lemma \ref{lem:lem4.4}, to bound the kernel of the map on the product of local cohomology groups.
\end{proof}

Below, for a $\Lambda$-module $M$ and $n\in \Z_{\ge 0}$, we denote $M^{\omega_n^+ = \omega_n^- = 0}$ simply by $M^{\omega_n^\pm = 0}$.
\begin{cor}\label{cor:corank-sel-one}
	Assume \eqref{ass:tama} and \eqref{ass:sel} hold. We have $\corank(\Sel^1(K_n,A)^{\omega_n^\pm=0})\le 1$.
\end{cor}
\begin{proof}
	The proof is due to Matar \cite[pp.~140-141]{matar-supersingular}, which we record here for completeness. As $\gcd(\omega_n^+,\omega_n^-) = X$ in $\Q_p[X]$, we have $A(X),B(X)\in \Z_p[X]$ and $m\in \Z_{\ge 0}$ such that $A\omega_n^+ + B\omega_n^- = p^m X$. It follows that
	\begin{align*}
		\corank_\fo(\Sel^1(K_n,A)^{\omega_n^\pm = 0}) \le \corank_\fo(\Sel^1(K_n,A)^{p^m(\gamma-1)=0}) = \corank_\fo(\Sel^1(K_n,A)^\Gamma).
	\end{align*}
	By Lemma \ref{lem:control-sel-one}, we conclude the last quantity is $\corank_\fo(\Sel^1(K,A))$, and thus is no greater than $\corank_\fo(\Sel(K,A))$, which is $\le 1$ by \eqref{ass:sel}.
\end{proof}

\subsection{Nontorsionness of $\cX^\pm_\infty$ via generalized Heegner modules}
\label{sec:appendix-nontorsion}
From now on we will always assume that $p$ is unramified in the local Hecke field $\fF$; this ensures that all the cyclotomic polynomials $\Phi_n(X)$ are irreducible. For $n\in \Z_{\ge 0}\cup \{\infty\}$, denote $\cX^\pm_n = \Sel_\pm(K_n,A)^\vee$. In this section we modify an argument of Matar to prove both $\Lambda$-modules $\cX^\pm_\infty$ are nontorsion. 

Recall we have assumed that $a_p=a_p(f)=0$, $p\nmid 2N\varphi(N)(k-1)!h_K$ and the weight of the modular form to be $k=2r\in 2\cdot \Z_{>0}$. For $n\in \Z_{\ge 0}\cup\{\infty\}$, let $\Psi(K_n)\subseteq H^1_\bff(K_n,T)$ be the image of the $p$-adic Abel--Jacobi map recalled in Definition \ref{def:AJ-image}. Since $\Psi(K_n)$ is an $\fo$-submodule of $H^1(K_n,T)$, we see that $\Psi(K_n)$ is finite free over $\fo$ since $H^1(K_n,T)$ is (using Lemma \ref{lem:lem4.4} again and the inclusion $A^{G_{K_n}}\subseteq A^{G_{K_{n,w}}}$ for $w\mid p$).

Next, let $z_{f,p^n}\in H^1(K[p^n],T)$ be the generalized Heegner classes recalled in \S\ref{sec:heegner}, where $K[p^n]$ is the ring class field of $K$ of conductor $p^n$. Recall we write $\cores_{K[1]/K}(z_{f,1})$ as $z_{f,K}$. For $n\in \Z_{\ge 0}$, denote $\alpha_n = \cores_{K[p^{n+1}]/K_n}(z_{f,p^{n+1}})\in \Psi(K_n)$, which in fact belongs to $\Sel(K_n,T)$ as explained in \S\ref{sec:heegner}. We have the following norm relations (see \cite[Lemma 4.1]{longo-vigni-kyoto} when $N^-=1$ and \cite[Lemma 4.1]{pati:forum-math} when $N^->1$):
\begin{enumerate}
	\item[(H1)] for $n\ge 2$, $\cores_{K_n/K_{n-1}}(\alpha_n) = -p^{k-2}\alpha_{n-2}$;
	\item[(H2)] $\cores_{K_1/K_0}\alpha_1 = p^{r-1}(p-1)/2\cdot\alpha_0$;
	\item[(H3)] $\alpha_0 = -2p^{r-1} z_{f,K}$.
\end{enumerate}
Throughout this section, we will assume the following linear non-incidence condition:
\begin{align}\tag{Inc.}\label{ass:incidence}
	\alpha_0 \ne 0,\ p\alpha_1\ne \frac{p^{r-1}(p-1)}{2}\alpha_0,\ p\alpha_2\ne -p^{k-2}\alpha_0.
\end{align}

In particular, since in our setting $H^1(K,T)[p^\infty] = 0$, $\alpha_0\ne 0$ if and only if $z_{f,K}\ne 0$, which is further equivalent to $z_{f,K}$ being nontorsion.

Define $\cH_n\subseteq \Psi(K_n)$ to be the $\fo$-span of $\alpha_n^\sigma$ for $\sigma$ ranging in $\gal(K_n/K)$. If $M$ is a $\fo$-module write $M_{\fF} = M\otimes_{\fo} \fF$. For $n\in \Z$, put $\varepsilon(n) = (-1)^n$. We prove
\begin{prop}\label{prop:structure-of-Hn}
	Suppose \eqref{ass:incidence} holds. Then the $\fF[\gal(K_n/K)]$-module $\cH_{n,\fF}$ is isomorphic to $\Lambda_{\fF}/(\omega_n^{\varepsilon(n)}(X))$.
\end{prop}
\begin{proof}
	By the norm relations (H1) and (H2), we have
		\begin{align*}
			\omega_n^{\varepsilon(n)}(X)\cdot \alpha_n \in \Z_p (\gamma-1)\cdot \alpha_0.
		\end{align*}
		As $(\gamma-1)\alpha_0=0$, we deduce that $\ann(\cH_{n,\fF})\supseteq (\omega_n^{\varepsilon(n)}(X))$ since the $\Lambda_{\fF}$-module $\cH_{n,\fF}$ is generated by $\alpha_n$. We now prove the reverse inclusion. For this, we choose a generator $F$ of $\ann(\cH_{n,\fF})$, which exists as $\Lambda_{\fF}$ is a principal ideal domain.
	
	We shall prove the reverse inclusion by induction. The case $n=0$ follows from the assumption $\alpha_0\ne 0$. For $n=1$, we have $\omega_n^{\varepsilon(n)} = \omega_1(X) = X^p-1$. Suppose $F\in \Lambda_{\fF}$ is a proper divisor of $\omega_1$, then $F = X$ or $F = \Phi_1(X)$. In the former scenario, we find $\gamma \alpha_1 = \alpha_1$, so
	\begin{align*}
		p^{r-1}(p-1)/2\cdot \alpha_0=\cores_{K_1/K}(\alpha_1) = p \alpha_1,
	\end{align*}
	contravening our hypothesis. Next, suppose $\Phi_1(X)\alpha_1=0$, which, by the norm relation, is equivalent to $p^{r-1}(p-1)/2\cdot\alpha_0 = 0$. Again this is ruled out as $\Psi(K)$ is free over $\fo$ and we assumed $\alpha_0\ne 0$. Thus for $n=1$, we must have $F = \omega_1(X)$.
	
	Now suppose we are done with $n<t$ for some $t\in \Z_{\ge 2}$. For $n=t$, first suppose for the sake of contradiction that $\Phi_t\nmid F$. Then we have $\omega_{t-1}(X)\cdot \alpha_t = 0$, namely $\alpha_t \in H^1(K_t,T)^{\Gamma_{t-1}} = H^1(K_{t-1},T)$ by Lemma \ref{lem:lem4.4}. Thus,
	\begin{align*}
		-p^{k-2}\alpha_{t-2} = \cores_{K_t/K_{t-1}}(\alpha_t) = p\alpha_t.
	\end{align*}
	Taking corestrictions on both sides, we find
	\begin{align*}
		\begin{cases}
			p\alpha_2 = -p^{k-2}\alpha_0, & \text{ if }t\equiv 0\bmod 2,\\
			p\alpha_1 = p^{r-1}(p-1)/2\cdot \alpha_0, & \text{ if }t\equiv 1\bmod 2.
		\end{cases}
	\end{align*}
	As both possibilities are excluded from our assumption, we conclude that $\Phi_t$ must divide $F$.
	
	Now, for any $\sigma\in \gal(K_n/K)$, we have
	\begin{align*}
		0=(p^{k-2}F/\Phi_t)(\Phi_t/p^{k-2})(\alpha_n^\sigma) = -(p^{k-2}F/\Phi_t)\alpha_{n-2}^\sigma.
	\end{align*}
	Thus $F/\Phi_t\in \ann(\cH_{t-2,\fF})$, which by inductive hypothesis implies that $\omega_{t-2}^{\varepsilon(t)}\mid F/\Phi_t$. It follows that $\omega_t^{\varepsilon(t)}\mid F$, as desired.
\end{proof}

\begin{cor}\label{cor:nontorsion}
	Assume \eqref{ass:incidence} holds. The $\Lambda$-modules $\cX_\infty^\pm$ are not torsion.
\end{cor}
\begin{proof}
	For $n\in \Z_{\ge 0}$, the norm relations imply that $\cH_n\otimes_{\fo} \fF/\fo \subseteq \Sel^{(-1)^n}(K_n,A)$. By Proposition \ref{prop:structure-of-Hn} we have an identification $\cH_n\otimes_{\fo} \fF/\fo\simeq (\fF/\fo)[X]/(\omega_n^{\varepsilon(n)})$. Thus $\cX_\infty^+$ admits a surjective map to $\fo[X]/(\omega_{2n}^{+})$. Now, if $\cX_\infty^+$ is $\Lambda$-torsion, then there exists $F\in \fo[[X]]$ such that $F$ annihilates $(\cH_{2n}\otimes \fF/\fo)^\vee$ for all $n\in \Z_{\ge 0}$. This forces $\omega_{2n}^+\mid F$ in $\Lambda_{\fF}$ for all such $n$, which is impossible. For the same reason, $\cX_\infty^-$ is also nontorsion.
\end{proof}

Now, for $n\in \Z_{\ge 0}$ we put $\cH(K_n)=\sum_{0\le m\le n}\cH_m$, $H(K_n) = \cH(K_n)\otimes_\fo \fF/\fo$, $H(K_n)^{\varepsilon(n)} = \cH_n\otimes_\fo \fF/\fo$ and $H(K_n)^{-\varepsilon(n)} = \cH_{n-1}\otimes_\fo \fF/\fo$. By the norm relations, we have
\begin{align}\label{eq:heegner-decompose}
	H(K_n) = H(K_n)^+ + H(K_n)^-.
\end{align}

\begin{prop}\label{prop:lower-bound}
	Assume \eqref{ass:tama}, \eqref{ass:sel} and \eqref{ass:incidence} hold. For $n\in \Z_{\ge 0}$, we have $\rank_{\fo}\cH(K_n)\ge p^n$ and $\rank_{\fo}\Psi(K_n)\ge p^n$.
\end{prop}
\begin{proof}
	The argument is due to Matar \cite[proof of Theorem 4.2]{matar-supersingular}. By \eqref{eq:heegner-decompose}, we have
	\begin{align*}
		\begin{split}
			\rank_{\fo}(\cH(K_n))&= \corank_{\fo}(H(K_n))\\
			&= \corank_{\fo}(H(K_n)^+) + \corank_{\fo}(H(K_n)^-) - \corank_{\fo}(H(K_n)^+\cap H(K_n)^-)\\
			&= p^n+1 - \corank_{\fo}(H(K_n)^+\cap H(K_n)^-),
		\end{split}
	\end{align*}
	where the last equality follows from Proposition \ref{prop:structure-of-Hn} and the identity $\deg(\omega_n^+) + \deg(\omega_n^-) = p^n+1$. As $H(K_n)^\varepsilon \subseteq \Sel^\varepsilon(K_n,A)^{\omega_n^\varepsilon=0}$ for $\varepsilon\in \{\pm 1\}$, we have the inclusion
	\begin{align*}
		H(K_n)^+\cap H(K_n)^- \subseteq \Sel^1(K_n,A)^{\omega_n^\pm = 0}.
	\end{align*}
	As such, by Corollary \ref{cor:corank-sel-one}, we conclude that
	\begin{align*}
		\rank_\fo(\cH(K_n)) \ge p^n + 1 - 1 = p^n.
	\end{align*}
\end{proof}

\subsection{Proof of Theorem \ref{thm:B}}
\label{sec:appendix-proof-main}

\begin{prop}\label{prop:pm-free}
	Assume \eqref{ass:tama}, \eqref{ass:sel} and \eqref{ass:incidence} hold. The $\Lambda$-modules $\cX_\infty^\pm$ are free of rank 1.
\end{prop}
\begin{proof}
	By Proposition \ref{prop:pm-control}, we find
	\begin{align*}
        \cX_\infty^\pm/(\gamma-1)\cX_\infty^\pm\simeq \Sel^\pm(K,A)^\vee \simeq \fo,
	\end{align*}
	where the last identity is by assumption \eqref{ass:sel}. As recalled in \S\ref{subsec:selmer-control}, both $\cX_\infty^\pm$ are finitely generated $\Lambda$-modules. Thus, $\cX_\infty^\pm$ are cyclic $\Lambda$-modules by Nakayama's lemma. Corollary \ref{cor:nontorsion} then implies that they are free.
\end{proof}

Next, we record the following analogue to \cite[Proposition 10.1]{kobayashi03}.
\begin{lem}\label{lem:kobayashi-surjection}
	Let $n\in \Z_{\ge 0}$. Let $j$ be the map
	\begin{align*}
		j: \Sel^+(K_n,A)^{\omega_n^+=0}\oplus \Sel^-(K_n,A)^{\omega_n^-=0} \to \Sel(K_n,A)
	\end{align*}
	induced from the inclusions $\Sel^{\pm}(K_n,A)\to \Sel(K_n,A)$, sending $(P_1,P_2)$ to $P_1+P_2$. Then $\im(j)\supseteq \Sel(K_n,A)_{\rm div}$.
\end{lem}
\begin{proof}
	The proof in \textit{loc.~cit.}~can be carried verbatim to our situation: Let $F,G\in \Z_p[[X]]$ and $m\in \Z_{\ge 0}$ be such that $F\omega_n^+ + G\tilde{\omega}_n^- = p^m$. If $P\in \Sel(K_n,A)_{\rm div}$, choose $Q\in \Sel(K_n,A)$ such that $p^m Q = P$. Then $P^+ = G\tilde{\omega}_n^-\cdot Q\in \Sel^+(K,A_n)^{\omega_n^+ = 0}$ and $P^- = F\omega_n^+\cdot Q\in \Sel^-(K,A_n)^{\omega_n^- = 0}$ are such that $P^+ + P^- = P$.
\end{proof}

\begin{proof}[Proof of Theorem \ref{thm:B}, part 1]
	We have proved (I) in Proposition \ref{prop:pm-free}. For (II), note that we have already the lower bound from Proposition \ref{prop:lower-bound}. To establish the upper bound, we recycle the following computation of Matar \cite{matar-supersingular}:
	\begin{align*}
		\begin{split}
			\corank_\fo(\Sel(K_n,A)) & = \corank_\fo(\Sel^+(K_n,A)^{\omega_n^+ = 0}) + \corank_\fo(\Sel^-(K_n,A)^{\omega_n^- = 0})\\
			&\quad - \corank_\fo(\Sel^+(K_n,A)^{\omega_n^+ = 0}\cap \Sel^-(K_n,A)^{\omega_n^- = 0})\\
			&= \rank_\fo(\cX_\infty^+/\omega_n^+\cX_\infty^+) + \rank_\fo(\cX_\infty^-/\omega_n^-\cX_\infty^-) - \corank_\fo(\Sel^1(K_n,A)^{\omega_n^\pm = 0})\\
			&= \rank_\fo(\Lambda/\omega_n^+) + \rank_\fo(\Lambda/\omega_n^-) - 1\\
			& = p^n.
		\end{split}
	\end{align*}
	Here, the first equality uses Lemma \ref{lem:kobayashi-surjection}, the second uses Lemma \ref{lem:sel-one-intersection}, the third uses Proposition \ref{prop:pm-free} and the inclusion $\fF/\fo\simeq \Sel(K,A)\hookrightarrow \Sel^1(K_n,A)^{\omega_n^\pm = 0}$.
	
	Concerning (III): note that Proposition \ref{prop:lower-bound} already shows that $\corank_\fo(\Psi(K_n)\otimes_\fo \fF/\fo)\ge p^n$. As $\Psi(K_n)\otimes_\fo \fF/\fo\subseteq \Sel(K_n, A)$, their divisible parts must coincide, whereby $\Sha_{\AJ}(K_n,A) = \Sha_{\BK}(K_n,A)$. The finiteness of $\Sha_\AJ(K_n,A)$ then follows from that of $\Sha_\BK(K_n,A)$.
\end{proof}

\begin{proof}[Proof of Theorem \ref{thm:B}, part 2]
	We now prove (IV) and (V) with the aid of results from \S\ref{sec:bdp}. First note that (V) is immediate from Corollary \ref{cor:bdp-sha-trivial}.(2) and (III). Concerning (IV): from Remark \ref{rem:matar3.2}.(ii) we have $\Sel(K_\infty,A)^\Gamma\subseteq \Sel^{\emptyset,\emptyset}(K,A) \simeq (\fF/\fo)^{\oplus 2}$. Hence, by Nakayama's lemma we have $\cX_\infty\simeq \Lambda^2/\mathfrak{a}$ for some $\Lambda$-submodule $\mathfrak{a}\subseteq \Lambda^2$. Now, dualizing the exact sequence
	\begin{align*}
		0\to \ilim_n \Sel^1(K_n,A)^{\omega_n^\pm = 0}\to \ilim_n\Sel^+(K_n,A)^{\omega_n^+ = 0}\oplus\ilim_n\Sel^-(K_n,A)^{\omega_n^-=0}\to \Sel(K_\infty,A),
	\end{align*}
	we find a map $\cX_\infty\to \Lambda^{\oplus 2}$ with cokernel isomorphic to a certain projective limit of $\fo$-modules of ranks bounded by 1 by Corollary \ref{cor:corank-sel-one}. This forces $\cX_\infty$ to have $\Lambda$-rank at least 2 and thus $\mathfrak{a} = 0$ (otherwise $\cX_\infty$ has a nontrivial annihilator), which gives (IV).
\end{proof}


\printbibliography
\end{document}